\documentclass[12pt]{amsart}

\usepackage{amsmath,amsbsy,amssymb,mathrsfs, amsthm, color}
\usepackage[normalem]{ulem}
\usepackage{graphicx}
\usepackage{latexsym}
\usepackage{epsfig}
\usepackage{amsfonts}
\usepackage[margin=2cm]{geometry}

\newtheorem{theorem}{Theorem}[section]
\newtheorem{prop}[theorem]{Proposition}

\newtheorem{cor}[theorem]{Corollary}

\numberwithin{equation}{section}
\theoremstyle{definition}

\newcommand{\bb}{\backslash}
\newcommand{\ord}{\mbox{ord}}
\newcommand{\diag}{{\rm diag}}
\newcommand{\PGL}{\rm{PGL}}
\newcommand{\GL}{\rm{GL}}
\newcommand{\Sp}{\rm{Sp}}
\newcommand{\SL}{\rm{SL}}
\newcommand{\GSp}{\rm{GSp}}
\newcommand{\PGSp}{\rm{PGSp}}
\newcommand{\St}{\rm{St}}
\newcommand{\oo}{\mathcal{O}}
\newcommand{\pp}{\mathcal{P}}
\newcommand{\I}{\mathcal{I}}
\newcommand{\J}{\mathcal{J}}

\newcommand{\B}{\mathcal{B}}

\newcommand{\Tr}{\operatorname{Tr}}
\newcommand{\triv}{{\mathbf1}}
\newcommand{\f}{\frac}

\makeatletter
  \newcommand\figcaption{\def\@captype{figure}\caption}
  \newcommand\tabcaption{\def\@captype{table}\caption}
\makeatother

\begin{document}

\title{The Zeta Functions of Complexes from $\Sp(4)$}

\author{Yang Fang, Wen-Ching Winnie Li and Chian-Jen Wang}

\address{Yang Fang\\ Department of Mathematics\\ The Pennsylvania State University\\
University Park, PA 16802} \email{\tt yuf1983@gmail.com}

\address{Wen-Ching Winnie Li\\ Department of Mathematics\\ The Pennsylvania State University\\
University Park, PA 16802 \\ and
 Division of Mathematics\\ National Center for Theoretical Sciences \\ Hsinchu, Taiwan} \email{\tt wli@math.psu.edu}

\address{Chian-Jen Wang\\ Department of Mathematics\\ National Central University\\
Chungli, Taiwan} \email{\tt cwang@math.ncu.edu.tw}

\thanks{The first author was partially supported by the NSF grant DMS-0801096, the second author is supported
by the NSF grants  DMS-0801096 and DMS-1101368,
and the third author is partially supported by the NSC grant 99-2115-M-008-001. Part of this work was done when the first author was visiting
the National Center for Theoretical Sciences in Hsinchu, Taiwan. She would like to thank NCTS for
its hospitality.}

\begin{abstract} Let $F$ be a  non-archimedean local field with a finite residue field. To a $2$-dimensional finite complex $X_\Gamma$ arising as the quotient of  the Bruhat-Tits building $X$ associated to $\Sp_4(F)$ by a  discrete torsion-free cocompact subgroup $\Gamma$ of $\PGSp_4(F)$, associate the zeta function $Z(X_{\Gamma}, u)$ which counts geodesic tailless cycles contained in the $1$-skeleton of $X_{\Gamma}$.  Using a representation-theoretic approach,
we obtain two closed form expressions for $Z(X_{\Gamma}, u)$ as a rational function in $u$. Equivalent statements for $X_{\Gamma}$ being a Ramanujan complex are given in terms of
vertex, edge, and chamber adjacency operators, respectively. The zeta functions of
such Ramanujan  complexes are distinguished by satisfying the Riemann Hypothesis.
\end{abstract}

\maketitle

\section{Introduction}
Ever since its introduction by Ihara \cite{Ih}, the zeta function of graphs has been well studied for more than 40 years. Defined similar to the zeta function of curves over finite fields, the zeta function of a graph $\mathbb{X}$ counts the number $N_n$ of tailless geodesic cycles in $\mathbb{X}$ of length $n$, and it also has an Euler product expression:
$$Z(\mathbb{X}, u) = \exp (\sum_{n \ge 1} \frac {N_n}{n} u^n) = \prod_{[C]} (1 - u^{l([C])})^{-1},$$
where the product is over equivalence classes $[C]$ of tailless primitive geodesic cycles $C$, and $l([C])$ is the length of a cycle in $[C]$.
The combined results of Ihara \cite{Ih} and Hashimoto \cite{Ha} combined show that $Z(\mathbb{X},u)$ is a rational function providing topological and spectral information of the graph, and it can be expressed in terms of the vertex adjacency operator $A$ and the edge adjacency operator $A_e$, respectively. When $\mathbb{X}$ is $(q+1)$-regular, its zeta function has the following form:
\begin{eqnarray}\label{graphzeta}
Z(\mathbb{X},u) = \frac{(1 - u^2)^{\chi(\mathbb{X})}}{\det(I - Au + qu^2I)} = \frac{1}{\det(I - A_e u)},
\end{eqnarray}
where $\chi(\mathbb{X})$ is the Euler characteristic of $\mathbb{X}$.
Furthermore, $Z(\mathbb{X},u)$ satisfies the Riemann Hypothesis in the sense that all non-trivial poles have the same absolute value if and only if $\mathbb{X}$ is spectrally optimal, called a Ramanujan graph in Lubotzky-Phillips-Sarnak \cite{LPS}.

Graphs are $1$-dimensional complexes. It is natural to search for zeta functions of finite higher dimensional complexes with similar features. When $q$ is a prime power, a finite $(q+1)$-regular graph can be identified with a finite quotient of the Bruhat-Tits tree $\B_2$ associated to $\SL_2$ over a nonarchimedean local field $F$ with $q$ elements in its residue field. Analogous higher dimensional complexes would be finite quotients of the Bruhat-Tits buildings associated to $p$-adic Chevalley groups; such a building is simply connected and the underlying group facilitates algebraic parametrizations of the geometric objects in the building and its finite quotients so that the whole setting is amenable to group representations. Further, a cycle in a quotient is called a geodesic if it lifts to a geodesic in the building, namely a straight line in an apartment. A zeta function will count tailless geodesic cycles contained in the $1$-skeleton of the complex.

A closed-form expression for the zeta function $Z(Y,u)$ of a two-dimensional finite
complex $Y$ which arises as a quotient of the building $\mathcal B_3$ associated to $\SL_3(F)$
by a discrete torsion-free co-compact subgroup $\Gamma$ in $\PGL_3(F)$ was obtained by Kang and Li \cite{KL} using combinatorial and group theoretic approach, and Kang-Li-Wang \cite{KLW} using representation-theoretic approach. It was shown in \cite{KL} and \cite{KLW} that $Z(Y, u)$ can be expressed as a rational function in two ways: one in terms of the two vertex adjacency operators $A_1$ and $A_2$ and the chamber adjacency operator $L_B$ on $Y$, and the other in terms of
the edge adjacency operator $L_{E}$ and its transpose $L_E^t$ on $Y$ as follows:
\begin{equation}\label{GL3zeta}
\begin{aligned}
Z(Y, u) &=  \frac{(1-u^3)^{\chi(Y)}}{\det(I-A_1u+qA_2u^2-q^3
u^3I)\det(I + L_Bu)} \\
&= \frac{1}{\det(I-L_{E}u)\det(I - L_E^tu^2)},
\end{aligned}
\end{equation}
where $\chi(Y)$ is the Euler characteristic of $Y$.
The concept of Ramanujan graphs was extended to Ramanujan complexes in Li \cite{Li1}, and by the definition therein, $Y$ is a Ramanujan complex if and only if all nontrivial zeros of $\det(I-A_1u+qA_2u^2-q^3 u^3I)$ have the same absolute value $q^{-1}$. In \cite{KLW}, equivalent statements for $Y$ being Ramanujan were given in terms of the eigenvalues of the edge operator $L_E$ and those of the chamber operator $L_B$, respectively.

The aim of this paper is to study the zeta functions of $2$-dimensional finite complexes associated to $\Sp_4(F)$.
More precisely, given a discrete torsion-free co-compact subgroup $\Gamma$ of $\PGSp_4(F)$, consider the zeta function of the $2$-dimensional complex $X_\Gamma$ arising as the quotient by $\Gamma$ of the Bruhat-Tits building $X$ associated to $\Sp_4(F)$. Since $X$ contains both special and non-special vertices, the group $\PGSp_4(F)$ does not act transitively on the vertices or edges of $X$, which is different from the action of $\PGL_3(F)$ on $\B_3$. Nonetheless, the vertices, edges and directed chambers of $X_\Gamma$ can be parametrized by certain cosets of $\PGSp_4(F)$.
Using representation theoretic approach analogous to \cite{KLW}, we prove that similar statements hold for $Z(X_\Gamma, u)$. Indeed, Theorems \ref{zeta1} and \ref{zeta2} combined give the following result.

\begin{theorem} Assume $\ord_{\pi}  \det (\Gamma) \subset 4 \mathbb Z$. The zeta function $Z(X_\Gamma, u)$ is a rational function expressed in two ways, the first in terms of the two vertex adjacency operators $A_1$ and $A_2$ and the chamber adjacency operator $L_I$, and the second in terms of the two edge adjacency operators $L_{P_1}$ and $L_{P_2}$, as follows:
\begin{equation}\label{GSp4zeta}
\begin{aligned}
Z(X_\Gamma, u) &= \frac{(1-u^2)^{\chi(X_{\Gamma})}(1-q^2 u^2)^{2N_p-N_{ns}}}{\det(I-A_1 u+qA_2 u^2-q^3 A_1 u^3+q^6I u^4)\det(I-L_I u)}\\
&= \frac{1}{\det(I-L_{P_1} u)\det(I-L_{P_2}u^2)},
\end{aligned}
\end{equation}
where $\chi(X_{\Gamma})$ is the Euler characteristic of $X_{\Gamma}$, and $2N_p$ (resp. $N_{ns}$) is the number of special (resp. non-special) vertices in $X_{\Gamma}$.
\end{theorem}

The vertices of $X_{\Gamma}$ are classified into three types: primitive special, non-primitive special, and non-special. The edges of $X_{\Gamma}$ are classified into two types: those connecting primitive and non-primitive special vertices are of type $1$, and those connecting special and non-special vertices are of type $2$. See \S2 for detailed descriptions of vertices, edges, chambers of $X_{\Gamma}$ and their parametrizations.
The operators $A_1$ and $A_2$ in Theorem 1.1 describe adjacency among special vertices; they can be interpreted as Hecke operators on the space of $\PGSp_4(\oo)$-invariant vectors in $L^2(\Gamma \backslash \PGSp_4(F))$ endowed with right regular actions by $\PGSp_4(F)$. Here $\oo$ denotes the ring of integers in $F$.
When restricted to an irreducible unramified representation $(\rho, V_{\rho})$ occurring in $L^2(\Gamma \backslash \PGSp_4(F))$, $\det(I-A_1 u+qA_2 u^2-q^3 A_1 u^3+q^6I u^4)$ in (\ref{GSp4zeta}) with $u$ replaced by $q^{-s}$ gives the  reciprocal of the degree four spin $L$-function $L(s-\frac{3}{2}, \rho)$.
The operator $L_{P_1}$ describes adjacency among directed edges of type $1$, while $L_{P_2}$ describes adjacency among type $2$ edges. See \S 3 for algebraic and combinatorial descriptions of these adjacency operators.

Extending the definition of Ramanujan complexes from $\SL_n(F)$ to $\Sp_4(F)$, we call a finite quotient $X_\Gamma$ Ramanujan if all infinite-dimensional irreducible representations occurring in $L^2(\Gamma \backslash \PGSp_4(F))$ containing nontrivial $\PGSp_4(\oo)$-invariant vectors are tempered.  Similar to the case of $\SL_3$, the Ramanujan condition can be re-interpreted in terms of eigenvalues of the operators acting on vertices, edges, and directed chambers, respectively; see \S7 for details. When this happens, the zeta function $Z(X_\Gamma, u)$ is said to satisfy the Riemann Hypothesis.

\begin{theorem} The following statements are equivalent:

$(1)$ $X_{\Gamma}$ is Ramanujan.

$(2)$ All the nontrivial zeros of $\det(I-A_{1}u+qA_{2}u^{2}-q^{3}A_{1}u^{3}+q^{6}Iu^{4})$ have absolute value $q^{-\frac{3}{2}}$.

$(3)$ All the nontrivial zeros of $\det(I - L_{P_{1}} u)$ have absolute values $q^{-\frac{3}{2}}$ or $q^{-1}$.

$(4)$ All the nontrivial zeros of $\det(I - L_{P_{2}} u)$ have absolute values $q^{\alpha}$, where $-2 \le \alpha \le -1$.

$(5)$ All the nontrivial zeros of $\det(I - L_{I} u)$ have absolute values $1$ or $q^{\beta}$, where $-1 \le \beta \le -\frac{1}{2}$.

\end{theorem}

It is interesting to compare the zeta function identities (\ref{graphzeta}), (\ref{GL3zeta}) and (\ref{GSp4zeta}) for a graph, a $2$-dimensional complex from $\B_3$, and a $2$-dimensional complex from $X$ to observe how the zeta identities evolve. Each determinant has a combinatorial meaning and a representation-theoretic interpretation. In these three cases, the determinant of operators on vertices occurring in the combinatorial zeta function, denoted by $P_v(u)$ for $u=q^{-s}$, is the reciprocal of a product of local $L$-factors attached to
unramified representations. Further, in the case of $\SL_2$ (resp.  $\SL_3$), for certain choices of $\Gamma$, the polynomial $P_v(u)$ is closely related to the zeta function of a curve (resp.  surface)  defined over $\mathbb F_q$, as shown in Shimura \cite[Theorem 7.11]{Shi} and Hashimoto \cite{Ha} for curves and Laumon-Rapoport-Stuhler \cite{LRS} and Li \cite{Li1, Li2} for surfaces.
It would be interesting to know whether similar phenomenon holds for $\Sp_4$ case. As the building for $\Sp_4(F)$ is the simplest of general case, it is hoped that our results will shed some light for combinatorial zeta functions in general.
\smallskip

The authors would like to thank Ralf Schmidt for helpful communications.

\section{Preliminaries and notation}

Let $F$ be a non-archimedean local field with ring of integers $\oo$ and maximal ideal $\pp$. Suppose the residue field $k=\oo / \pp$ has finite order $q$. Fix a uniformizer $\pi$ in $\pp$. Most of the results in this section can be found in Garrett \cite{Ga}, Setyadi \cite{Se1}, and Shemanske \cite{Sh}.

\subsection{The groups $\Sp_4(F)$ and $\GSp_4(F)$}

Let $V$ be a 4-dimensional vector space over $F$, equipped with the non-degenerate alternating bilinear form $\langle\ ,\ \rangle$, given on the standard basis $\{e_1, e_2, f_1, f_2\}$ of $V$ according to
\begin{equation}\label{e:bilinear}
\begin{split}
\langle e_1, f_2 \rangle = \langle e_2, f_1 \rangle &=1, \text{  and} \\
\langle e_i, f_j \rangle = \langle e_i, e_j \rangle &= \langle f_i, f_j \rangle =0  \text{ for } 1 \le i= j \le 2.
\end{split}
\end{equation}
An ordered basis of $V$ satisfying the above conditions
is called a symplectic basis.

The linear transformations on $V$ preserving the bilinear form $ \langle \ ,\  \rangle$ form the symplectic group $\Sp_4(F)$, which can be identified with the subgroup of elements $g$ in $ \GL_4(F)$ satisfying $g^t J g=J$, where
$$J=\left( \begin{matrix} 0&0&0&1\\ 0&0&1&0\\ 0&-1&0&0\\ -1&0&0&0 \end{matrix}\right).$$
Those linear transformations preserving $ \langle \ ,\  \rangle$ up to a scalar multiple form the group $\GSp_4(F)$ of symplectic similitudes; it consists of elements $g \in \GL_4(F)$ that satisfy
$g^t J g=\lambda(g) J$ for some $\lambda(g)$ in $F^\times$. The center $Z$ of $\GSp_4(F)$ has elements of the form $\diag(z, z, z, z)$
with $z \in F^\times$. The projective symplectic similitude group $\PGSp_4(F)$ is defined as the quotient of $\GSp_4(F)$ by $Z$.

\subsection{The building $X$ associated to $\Sp_4(F)$}
We begin by recalling definitions concerning lattices.
A lattice $L$ inside V is a rank four free $\oo$-module. It is called {\it primitive} if $\langle L, L \rangle \subset \oo$ and the bilinear form $ \langle \ ,\  \rangle$ induces a non-degenerate alternating form on the vector space $L/{\pi L}$ over $k$. A lattice $L'$ is {\it incident} to a primitive lattice $L$ if $\pi L \subset L' \subset L$ and $\langle L', L' \rangle \subset \pp$. Two lattices $L$ and $L'$ are {\it homothetic} if $L'=\alpha L$ for some $\alpha \in F^\times$; denote by $[L]$ the homothety class of $L$.

\subsubsection{Vertices, edges, and chambers}
The Bruhat-Tits building $X$ of the symplectic group $\Sp_4(F)$ is a two dimensional simplicial complex, which can be realized as a flag complex associated to certain incidence geometry. Under this realization, the vertices of $X$ are the homothety classes of primitive lattices and
those incident to primitive ones (cf. \cite{Ga}). Two vertices in $X$ are connected by an edge if they can be represented by two lattices $L'$ and $L''$ with one contained in the other and there is a primitive lattice $L$ such that $\pi L \subset L', L'' \subset L$.
Three vertices form a chamber if they can be represented by lattices $L$, $L'$, and $L''$ so that $L$ is primitive
and $\pi L \subset L' \subset L'' \subset L$. Each chamber corresponds to a maximal simplex in $X$. A symplectic basis $\{u_1, u_2, w_1, w_2\}$ of $V$ specifies an apartment of $X$ in which all the vertices can be expressed in terms of $[\oo \pi^{a_1}u_1+\oo \pi^{a_2}u_2+\oo \pi^{b_1}w_1+\oo \pi^{b_2}w_2]$ for some $a_i, b_i \in \mathbb{Z}$ (cf. \cite[Lemma 3.1]{Se1}). The building $X$ is the union of all such apartments.

Let $L_0$ be the lattice $\oo e_1 + \oo e_2 + \oo f_1 +\oo f_2$, and $L_3=\oo e_1+\oo \pi e_2 +\oo \pi f_1 +\oo \pi f_2$ and $L_2=\oo e_1+\oo e_2 +\oo \pi f_1 +\oo \pi f_2$ be two sublattices. The homothety classes $[L_0]$, $[L_3]$, $[L_2]$ are vertices in $X$, and they lie in
the same apartment specified by the standard basis $\{e_1, e_2, f_1, f_2\}$. As $\pi L_0 \subset L_1 \subset L_2 \subset L_0$ and $L_0$ is primitive, these three vertices form a chamber $C$ which we fix as the fundamental chamber. The sides of $C$ form an isosceles right triangle with the edge connecting $[L_0]$ and $[L_2]$ as the hypotenuse. Figure 1 shows the fundamental apartment $\Sigma_0$ determined by
$\{e_1, e_2, f_1, f_2\}$. For brevity, we write $[a_1, a_2, b_1, b_2]$ for the homothety class $[\oo \pi^{a_1}e_1+\oo \pi^{a_2}e_2+\oo \pi^{b_1}f_1+\oo\pi^{b_2}f_2]$.

\begin{figure}
\begin{center}
\includegraphics{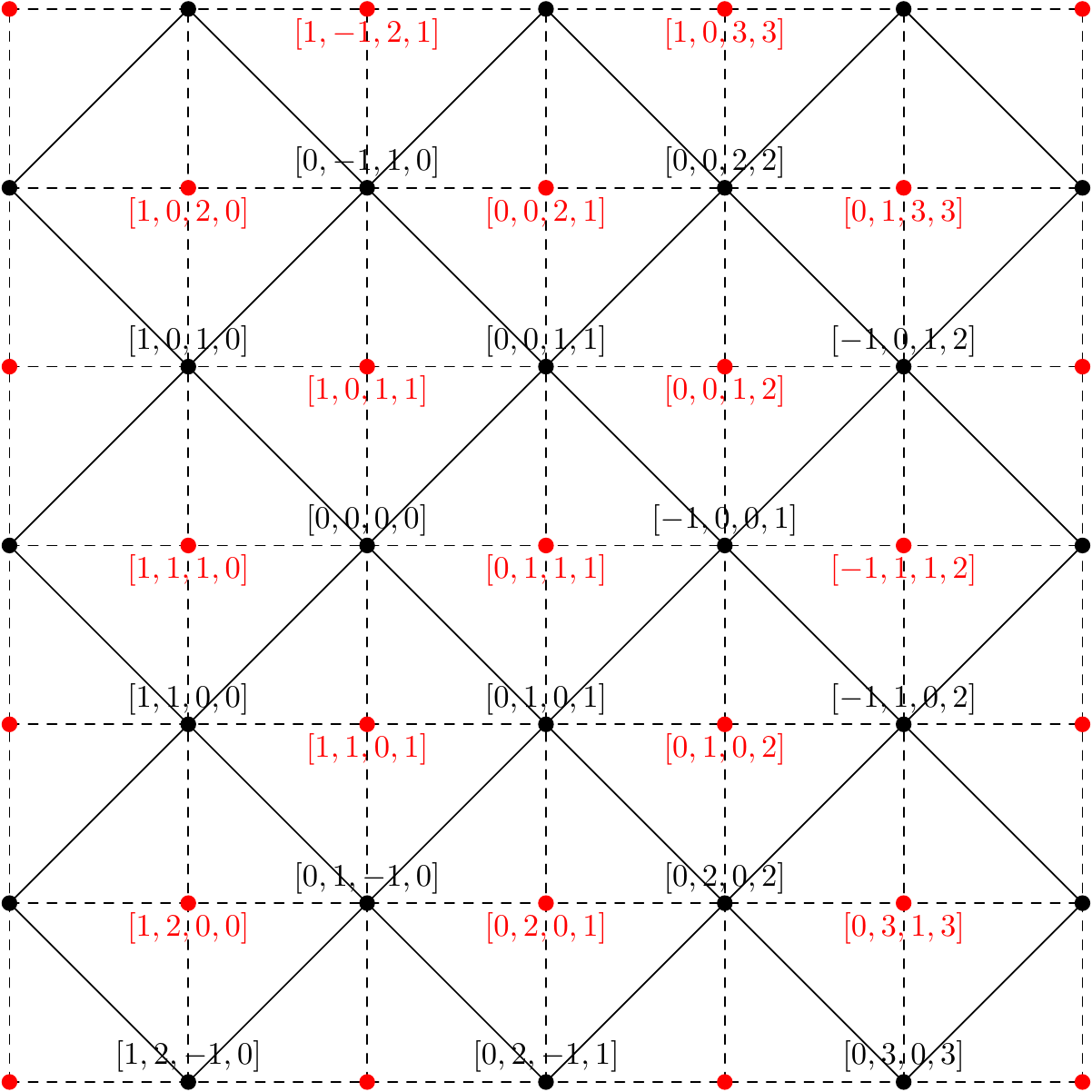}
\caption{Apartment $\Sigma_0$}
\end{center}
\end{figure}

\subsubsection{Types of vertices and edges}
If $[L]$ is a homothety class of lattices in $V$, then $[L]=[gL_0]$ for some $g \in \GL_4(F)$.  Define the type of $[L]$ to be
$\ord_{\pi}(\det g)$ modulo 4, which is independent of the choices of $g$ and $L$. It follows directly from the definition that if $g \in \GL_4(F)$ and $[L]$ has type $i$
modulo 4, then $g[L]:=[gL]$ has type $i+\ord_{\pi}(\det g)$ modulo 4.

When $[L]$ is a vertex in $X$, it can only have types 0, 2, or 3 modulo
4, so the vertices in $X$ are partitioned into three types.
Type 0 vertices, represented by primitive lattices, are classified as primitive vertices; vertices of other types are non-primitive. Both type 0 and type 2 vertices are special vertices (see Tits \cite{Ti} for precise definition), while type 3 vertices are non-special ones. None of the vertices with the same type are adjacent to each other. The symplectic group $\Sp_4(F)$ acts on vertices of $X$ by left translation; it is type-preserving.

If the vertex $[L]$ lies in the apartment specified by the symplectic basis $\{u_1, u_2, w_1, w_2\}$ so that it is written as $[\oo \pi^{a_1}u_1+\oo \pi^{a_2}u_2+\oo \pi^{b_1}w_1+\oo \pi^{b_2}w_2]$, then $[L]$ is special if and only if $a_1+b_2=a_2+b_1$; while $[L]$ is
primitive if and only if $a_1+b_2=a_2+b_1=0$. In particular,
$[L_0]=[0,0,0,0]$ is primitive, $[L_3]=[0,1,1,1]$ is non-special, and $[L_2]=[0,0,1,1]$ is special but not primitive. Further, the type of the vertex $[L_i]$ is $i$ for $i = 0,2,3$.

An edge in $X$ is said to be of type 1 if it connects two special vertices (one of
them must be primitive while the other is not); otherwise it connects a special vertex with a non-special one, and is of type 2. Thus a chamber contains one type $1$ edge and two type $2$ edges.

\subsection{Subgroups of $\GSp_4(F)$}
From now on, let $G=\GSp_4(F)$. We shall make $G$ act on $X$ and parametrize the simplices in $X$ algebraically.  Reviewed below are the subgroups
of $G$ we are concerned with.

The Borel subgroup $B$ of $G$ consists of upper triangular matrices in $G$, so an element in $B$ has the form $$\left(\begin{matrix}*&*&*&*\\&*&*&*\\&&*&*
\\&&&*\end{matrix}\right).$$
The Siegel parabolic subgroup $P$ and the Klingen parabolic subgroup $Q$ in $G$ have the forms
$$P=\left\{\left(\begin{matrix} *&*&*&*\\ *&*&*&* \\&&*&*\\&&*&*\end{matrix}\right)\right\} \text{ and } Q= \left\{\left(\begin{matrix}*&*&*&*\\&*&*&*\\&*&*&*\\&&&*\end{matrix}\right)\right\},$$
respectively. The Weyl group of $G$ has order 8, generated by $$s_1=\left(\begin{matrix} &1&&\\1&&&\\&&&1\\&&1&\end{matrix}\right) \text{ and }
s_2=\left(\begin{matrix} 1&&&\\&&1&\\&-1&&\\&&&1\end{matrix}\right).$$

There are two maximal compact subgroups in $G$ up to conjugacy. One is the group $K=\GSp_4(\oo)$, the other is the paramodular group $P_{02}$ consisting of  $g\in G$ such that
 $$ g\in \left(\begin{matrix}\oo & \oo & \oo &{\pp}^{-1} \\ \pp &\oo&\oo&\oo \\ \pp&\oo&\oo&\oo \\ \pp&\pp&\pp&\oo \end{matrix}\right) \text{ and } \det g \in {\oo}^\times. $$ The Iwahori subgroup $I$ consists of elements $g \in K$ which are upper triangular modulo $\pp$. The Siegel congruence subgroup $P_1=I \cup Is_1I$ and the Klingen subgroup
$P_2=I \cup I s_2 I $ are both subgroups in $K$. Modulo $\pp$, $P_1$ becomes the Siegel parabolic subgroup and $P_2$ the Klingen parabolic subgroup of $\GSp_4(k)$.

\subsection{Parametrization of the simplices in $X$}
Let $G_0$ be the index-2 subgroup of $G$ consisting of elements $g$ with $\ord_{\pi}(\det g) \equiv 0$ mod 4.  Put
$$\tau=\left(\begin{matrix} &&1& \\ &&&1 \\ \pi&&& \\&\pi&& \end{matrix}\right),$$
then $\tau^2 \in Z$, and $G=G_0 \cup G_0 \tau$. The group $G_0$ acts on the vertices of $X$ and preserves types,
while the element $\tau$ interchanges the primitive special vertices and the type 2 special ones. Note that
the lattice $L_1 = \pi^{-1}( \tau L_3)$ has type $1$ with
$\oo$-basis $\{e_1, e_2, f_1, \pi f_2 \}$, thus $\tau[L_3]=[L_1]$ is not a vertex of $X$; so
$\tau$ does not preserve non-special vertices. Define the dual lattice of $L$ to be $$L^*=\{v \in V | \langle v, w \rangle \in \oo \text{ for all } w \in L\}.$$
Note also that $[L_1]=[L_3^*]$, and $P_{02}Z=\tau P_{02}Z \tau^{-1}$ stabilizes both $[L_3]$ and $[L_3^*]$.

\subsubsection{Vertices} A vertex $[L]$ is called {\it self-dual} if $[L]=[L^*]$. It was shown in \cite[Corollary 3.4]{Sh} that self-dual vertices are precisely the special vertices. If $[L]=[gL_3]$ is a non-special
vertex with $g$ in $G_0$, then $[L^*]=[g\tau L_3]=[gL_3^*]$. We shall regard each vertex in $X$ labelled by $[L]$ and its dual $[L^*]$ so that $G$ acts transitively on special vertices (cf. \cite[Proposition 1.4]{Se2}) and non-special vertices respectively; in particular $\tau$ permutes $[L_0]$ and $[L_2]$, and stabilizes $[L_3]$. Since the stabilizer of $[L_0]$ in $G$ is $KZ$, the left cosets $\{gKZ\}_{g \in G}$ parametrize all the special vertices in $X$. More specifically, the cosets $\{g KZ\}_{g \in G_0}$ parametrize primitive vertices and the cosets $\{g\tau KZ\}_{g \in G_0}$ parametrize type 2 special vertices.
Let $\bar{P}_{02} = P_{02} \cup P_{02}\tau $ be the subgroup in $G$ generated by $P_{02}$ and $\tau$,  then the non-special vertices in $X$ are parametized by $\{g\bar{P}_{02}Z\}_{g \in G_0}$.

\subsubsection{Edges and chambers}
Edges between special vertices are of type 1, and those between special and non-special vertices are of type 2.
Each undirected edge gives rise to two directed edges. Denote by $E_1$ the directed type $1$ edge from $[L_0]$ to $ [L_2]$, and $E_2$ the directed type $2$ edge from $[L_0]$ to $[L_3]$.  Modulo $Z$, the Siegel congruence subgroup $P_1=K\cap \tau K{\tau}^{-1}$ is the  stabilizer of $E_1$, and hence the cosets $\{gP_1Z\}_{g \in G}$ parametrize directed type 1 edges in $X$. Similarly, up to $Z$,
the Klingen congruence subgroup $P_2=K\cap P_{02} = K\cap \bar{P}_{02}$ is the stabilizer of $E_2$ so that the cosets $\{gP_2Z\}_{g \in G}$ parametrize all directed type 2 edges from special vertices to non-special vertices, or simply all undirected type $2$ edges. Up to $Z$, the Iwahori subgroup $I=P_1 \cap P_2$ is the stabilizer  of the chamber $C$ along with the directed edge $E_1$, and the cosets $\{gIZ\}_{g \in G}$ parametrize the directed chambers, that is, chambers with oriented type $1$ edges. In particular, $\tau IZ$ represents the chamber $C$ whose type $1$ edge is opposite to $E_1$.

\subsection{The number of neighboring vertices}
We compute the number of neighbors of a vertex in $X$ of a given type. Because of the group action, it suffices to count these for the vertices $[L_0]$, $[L_2]$ and $[L_3]$. First, the type $2$ neighbors $[L_2']$ of $[L_0]$ can be identified with the type $2$ lattices $L_2'$ satisfying $\pi L_0 \subset L_2' \subset L_0$ such that $\langle L_2', L_2' \rangle = \pp$. The total number, by Proposition 1.7 in \cite{Se2}, is $q^3+q^2+q+1$. Apply $\tau$ to $[L_0]$ to conclude that $[L_2]$ has $q^3+q^2+q+1$ type $0$ neighbors. This shows that out of a special vertex there are $q^3+q^2+q+1$ type $1$ edges. The type $3$ neighbors of $[L_0]$ are represented by type $3$ lattices $L_3'$ satisfying $\pi L_0 \subset L_3' \subset L_0$; the total number is equal to the number of lines in $L_0/\pi L_0$, namely $q^3+q^2+q+1$. This is also the number of type $3$ neighbors of $[L_2]$ by applying $\tau$. Hence out of each special vertex there are $q^3+q^2+q+1$ type $2$ edges. Note that the type 2 neighbors $[L_2']$ of $[L_3]$ are represented by type $2$ lattices $L_2'$ containing $L_3$ and contained in $L_1$. So they correspond to the lines in the $2$-dimensional space $L_1/L_3$. This shows that each type $3$ vertex has $q+1$ type $2$ neighbors, and hence $q+1$ type $1$ neighbors by symmetry. So out of each non-special vertex there are $2(q+1)$ type $2$ edges and no type $1$ edges. We record the above discussion in the
following proposition.

\begin{prop}\label{vertex0} Each special vertex in $X$ is adjacent to $q^3+q^2+q+1$ special vertices of different type and $q^3+q^2+q+1$ non-special vertices, while each non-special vertex in $X$ is adjacent to $2(q+1)$ special vertices, half primitive and half type $2$ special.
\end{prop}

\subsection{The finite quotient $X_{\Gamma}$}
Let $\Gamma$ be a discrete  torsion-free co-compact-mod-center subgroup of $G$. The existence of $\Gamma$ follows from Borel and Harder \cite{BH}. Assume that $ \Gamma \subset G_0$. Then $\Gamma$ acts freely on $X$ and is type-preserving. Consider the finite quotient complex $\Gamma \bb X$,  denoted by $X_{\Gamma}$. It is a $(q+1)$-regular complex in the sense that each edge is contained in exactly $q+1$ chambers (cf.  \cite[Proposition 3.2]{Se1}).

It follows from the algebraic parametrization that the numbers of special vertices and type 2 edges in $X_{\Gamma}$ are equal to the cardinality of the double cosets in $\Gamma \bb G/KZ$ and $\Gamma \bb G/{P_2}Z$ respectively. The numbers of non-special vertices, type 1 edges, and chambers in $X_{\Gamma}$ are one-half the numbers of the double cosets in $\Gamma \bb G/P_{02}Z$, $\Gamma \bb G/{P_1}Z$, $\Gamma \bb G/IZ$ respectively.

\begin{prop}\label{vertex} Let $N_p$, $N_s$, $N_{ns}$ denote the numbers of primitive, type 2 special, and non-special vertices in $X_{\Gamma}$ respectively. Then the ratio $N_p: N_s: N_{ns}=1:1:(q^2+1)$.
\end{prop}

\begin{proof}
As $\Gamma$ acts freely on $X$, two distinct edges sharing a common vertex will not be identical modulo $\Gamma$. So according to Proposition \ref{vertex0},
there are $q^3+q^2+q+1$ type 1 edges and the same number of type 2 edges out of each special vertex in $X_{\Gamma}$, and there are $2(q+1)$ type 2 edges out of
each non-special vertex in $X_{\Gamma}$.

Recall that each type 1 edge connects a primitive vertex with a types 2 special vertex. By considering the number of type $1$ edges in $X_\Gamma$ we get $N_p (q^3+q^2+q+1) = N_s(q^3+q^2+q+1)$, which shows $N_p=N_s$. Similarly, from the number of type $2$ edges in $X_\Gamma$ we get $(N_p + N_s)(q^3+q^2+q+1) = 2(q+1)N_{ns}$, which yields $N_{ns} = (q^2+1)N_p$.
\end{proof}

\section{Adjacency operators on vertices, edges, and chambers}
In this section we introduce adjacency operators for special vertices, directed type 1 edges, type 2 edges, and directed chambers.  They will be defined on the building $X$ first, and then passed to the finite quotient $X_{\Gamma}$.

\subsection{Vertex adjacency operators}
Recall that the special vertices in $X$ are parametrized by $\{gKZ\}_{g \in G}$. On the space $C(G/KZ)$ of complex-valued functions on $G/KZ$,
define $A_1$ to be the operator sending $f \in C(G/KZ)$ to
$$A_1 f(gKZ)=\sum_{g_iK\in K\diag(1,1,\pi,\pi)K/K} f(gg_i KZ);$$
and $A_2$ the operator sending $f$ to
$$A_2 f(gKZ)=(q^2+1)f(gKZ)+ \sum_{h_jK\in K\diag(1, \pi, \pi, \pi^2)K/K} f(gh_jKZ).$$
These are the Hecke operators associated to the double $K$-cosets represented by $\diag(1,1,\pi,\pi)$ and $\diag(1,\pi,\pi,\pi^2)$ respectively.

The coset decompositions that appear in the definition above can be explicitly written out as follows:
\begin{multline}\label{A_1}
K\diag(1,1,\pi,\pi)K=\bigsqcup_{i \in \I} g_iK=\bigsqcup_{a, b, c\in \oo/\pp}\left(\begin{matrix} \pi&0&b&a\\ 0&\pi&c&b\\ 0&0&1&0\\ 0&0&0&1 \end{matrix}\right)K   \\
\bigsqcup_{\alpha, \beta \in \oo/\pp} \left(\begin{matrix}\pi&-\alpha&0&\beta\\0&1&0&0\\0&0&\pi&\alpha\\0&0&0&1\end{matrix}\right)K
\bigsqcup_{\gamma \in \oo/\pp}\left(\begin{matrix}1&0&0&0\\0&\pi&\gamma&0\\0&0&1&0\\0&0&0&\pi\end{matrix}\right)K
\bigsqcup \left(\begin{matrix}1&0&0&0\\0&1&0&0\\0&0&\pi&0\\0&0&0&\pi\end{matrix}\right)K;
\end{multline}

\begin{multline}\label{A_2}
K\diag(1, \pi, \pi, \pi^2)K=\bigsqcup_{j \in \J} h_jK=\bigsqcup_{a,b \in \oo/\pp, c\in \oo/{\pp}^2} \left(\begin{matrix}\pi^2&-\pi a&\pi b&c\\0&\pi&0&b\\0&0&\pi&a\\0&0&0&1\end{matrix}\right)K  \\
\bigsqcup_{u \in \oo/\pp, v\in \oo/{\pp}^2} \left(\begin{matrix}\pi&0&u&0\\0&\pi^2&v&\pi u\\0&0&1&0 \\0&0&0&\pi\end{matrix}\right)K
\bigsqcup_{w\in \oo/\pp}\left(\begin{matrix}\pi&-w&0&0\\0&1&0&0\\0&0&\pi^2&\pi w\\0&0&0&\pi\end{matrix}\right)K \\
\bigsqcup_{\alpha \in (\oo/\pp)^\times }\left(\begin{matrix}\pi&0&0&\alpha\\0&\pi&0&0\\0&0&\pi&0\\0&0&0&\pi\end{matrix}\right)K
\bigsqcup_{\beta \in \oo/\pp, \gamma \in (\oo/\pp)^\times} \left(\begin{matrix}\pi&0&\beta&\gamma\\0&\pi&\frac{\beta^2}{\gamma}&\beta\\0&0&\pi&0\\0&0&0&\pi\end{matrix}
\right)K
\bigsqcup \left(\begin{matrix}1&0&0&0\\0&\pi&0&0\\0&0&\pi&0\\0&0&0&\pi^2\end{matrix}\right)K.
\end{multline}

Identify each vertex with its parametrization. Then $A_1$, $A_2$ can be regarded as operators on special vertices. To see how they act, it suffices to
check on the special vertex $[L_0]=KZ$. For $A_1$, there are $q^3+q^2+q+1$ left $K$-cosets in the decomposition (\ref{A_1}). So each special vertex $gKZ$
gives rise to $q^3+q^2+q+1$ special vertices $\{gg_iKZ\}_{i \in \I}$, which we refer to as the neighboring vertices of $gKZ$ with respect to $A_1$.
The neighboring vertices of $[L_0]$ with respect to $A_1$ that lie in the fundamental apartment $\Sigma_0$ are obtained from the $K$-cosets  $\diag(\pi,\pi,1,1)K$, $\diag(\pi,1,\pi,1)K$, $\diag(1,\pi,1,\pi)K$, and $\diag(1,1,\pi,\pi)K$ in (\ref{A_1}), labeled as vertices $[1,1,0,0], [1,0,1,0], [0,1,0,1]$, and $[0,0,1,1]$ in $\Sigma_0$ (cf. Figure 1).
These are all of the special vertices in $\Sigma_0$ adjacent to $[L_0]$.
If  $\Sigma$ is a different apartment  containing $[L_0]$, then $\Sigma=g'\Sigma_0$ for some $g'\in K$. So the left translation by $g'$ maps
 the type $2$ neighbors of $[L_0]$ in $\Sigma_0$ to those in $\Sigma$. This shows that the special vertices $\{g_iKZ\}_{i \in \I}$ are distinct and all adjacent to $[L_0]$. As there are
exactly $q^3+q^2+q+1$ special vertices adjacent to a given one, $A_1$ can be interpreted as the vertex adjacency operator on special vertices of $X$.

Similar argument applies to the operator $A_2$. The coset decomposition in (\ref{A_2}) shows that the neighboring vertices
 of $[L_0]$ with respect to $A_2$ that lie in the apartment $\Sigma_0$ are parametrized by $\diag(\pi^2, \pi, \pi, 1)KZ$, $\diag(\pi, \pi^2, 1, \pi)KZ$, $\diag(\pi, 1, \pi^2, \pi)KZ$, and $\diag(1, \pi, \pi, \pi^2)KZ$, which are the vertices $[1,0,0,-1], [0,1,-1,0],[0,-1,1,0]$, and $[-1,0,0,1]$ in Figure 1. These are the special vertices in $\Sigma_0$ that have edge distance two from $[L_0]$. So $A_2$ can be interpreted as the edge distance $2$ adjacency operator  plus $(q^2+1)$ times the identity operator on special vertices of $X$.

\subsection{Edge adjacency operators}
The directed type 1 edges in $X$ are parametrized by $G/{P_1}Z$. We define $L_{P_1}$ to be the operator on the space $C(G/{P_1}Z)$, sending $f$ in $C(G/{P_1}Z)$ to
$$L_{P_1}f(gP_1Z)=\sum_{g_i P_1 \in P_1 \diag(1,1,\pi,\pi)P_1/{P_1}} f(gg_i P_1Z).$$ The relevant coset decomposition is
\begin{equation}\label{L_{P_1}}
P_1\diag(1,1,\pi,\pi)P_1 = \bigsqcup_{i \in \I'} g_i P_1 =\bigsqcup_{a,b,c \in \oo/\pp}\left(\begin{matrix}1&0&0&0\\0&1&0&0\\\pi b &\pi a&\pi&0\\\pi c&\pi b&0&\pi\end{matrix}\right)P_1.
\end{equation}

To facilitate the description of $L_{P_1}$, following \cite{Se1}, we call two vertices $v_1$, $v_2$ in $X$ {\it close}
if there are two edge sharing chambers $C_1$ and $C_2$ of $X$ such that $v_i \in C_i$ for $i = 1, 2$ and $v_i \notin C_j$ for $i \ne j$.
Recall the directed type $1$ edge $E_1$ from $KZ$ to $ \tau KZ$ (namely the edge $[0,0,0,0] \to [0,0,1,1]$ in Figure 1) in the fundamental chamber; it is parametrized by $P_1Z$. It has $q^3$ neighboring edges with respect to $L_{P_1}$, $\{g_i P_1Z\}_{i \in \I'}$, given by (\ref{L_{P_1}}).
Of these,  $\diag(1, 1, \pi, \pi)P_1Z$ is the only one lying in the fundamental apartment $\Sigma_0$; it is the type 1 directed edge
$\sigma KZ \to \diag(1,1,\pi^2,\pi^2)KZ$, or $[0,0,1,1] \to [0,0,2,2]$ in Figure 1. So the directed edge $P_1Z$ and its neighboring edge $\diag(1,1,\pi, \pi)P_1Z$ are adjacent segments of the same directed straight line in $\Sigma_0$.
Other neighboring edges of $P_1Z$ with respect to $L_{P_1}$ are similarly seen from the apartments containing $P_1Z$.
In general, the $q^3$ neighboring edges with respect to $L_{P_1}$ of the given directed type $1$ edge $gKZ \to g'KZ$ can be described as $g'KZ \to g''KZ$, where $g''KZ$ is a special vertex adjacent to $g'KZ$, but not close to $gKZ$.  Under this interpretation we regard $L_{P_1}$ as the adjacency operator on directed type $1$ edges.

The type 2 edges in $X$ are parametrized by $G/{P_2}Z$, directed from special vertices to non-special vertices. Define $L_{P_2}$ to be the operator on the space $C(G/{P_2}Z)$, sending $f \in C(G/{P_2}Z)$ to
$$L_{P_2}f(gP_2Z)=\sum_{h_j P_2 \in P_2 \diag(1, \pi, \pi, \pi^2)P_2/{P_2}} f(gh_j P_2Z).$$
The relevant coset decomposition here is
\begin{equation}\label{L_{P_2}}
P_2 \diag(1, \pi, \pi, \pi^2) P_2=\bigsqcup_{j \in \J'} h_j P_2 =\bigsqcup_{a,b \in \oo/\pp, c\in \oo/{{\pp}^2}} \left(\begin{matrix}1&0&0&0\\-\pi a&\pi&0&0\\
\pi b&0&\pi&0 \\ \pi c&\pi^2 b &\pi^2 a&\pi^2 \end{matrix}\right)P_2.
\end{equation}

The neighboring edges with respect to $L_{P_2}$ can be described as follows. Similar to $L_{P_1}$, one observes that the type 2 directed edge $P_2Z$ (which is the edge $E_2$ of the fundamental chamber and given by $[0,0,0,0] \to [0,1,1,1]$ in Figure 1) has only one neighbor $\diag(1, \pi, \pi, \pi^2)P_2Z$ in $\Sigma_0$, given by $[-1,0,0,1] \to [-1,1,1,2]$ in Figure 1. So in a fixed apartment, a
directed type 2 edge and its neighbor are part of the same directed straight line. In the building $X$, a type $2$ directed edge $gKZ \to g_1 \bar{P}_{02}Z$ has $q^4$ neighboring edges with respect to $L_{P_2}$, and they are characterized as $g'KZ \to g_1' \bar{P}_{02}Z$, where the special vertex $g'KZ$ is adjacent to $g_1 \bar{P}_{02}Z$ and close to $gKZ$, while the non-special vertex $g_1'\bar{P}_{02}Z$ is adjacent to $g'KZ$ but not close to $g_1 \bar{P}_{02}Z$. In this sense $L_{P_2}$ is the adjacency operator on type $2$ edges.

\subsection{Chamber adjacency operator}
The directed chambers in $X$ are parametrized by $G/IZ$. Let $t=\left(\begin{matrix}1&0&0&0\\0&0&1&0\\0&-\pi&0&0\\0&0&0&\pi\end{matrix}\right)$.
We define $L_I$ to be the Iwahori-Hecke operator on the space $C(G/IZ)$, given by
$$ L_I f(gIZ)=\sum_{g_jI \in I t I/I} f(gg_j IZ).$$

The coset decomposition in this case is
\begin{equation}\label{L_I}
I t I =\bigsqcup_{j \in \J''} g_j I =\bigsqcup_{a,b \in \oo/\pp} \left(\begin{matrix}1&0&0&0\\0&0&1&0\\\pi b&-\pi &0&0\\\pi a&0&\pi b&\pi \end{matrix}\right)I.
\end{equation}

The operator $L_I$ can be interpreted as a chamber adjacency operator on directed chambers, with the adjacency condition described below. The directed chamber parametrized by $IZ$ is the fundamental chamber $C$ with the directed edge $[0,0,0,0]\to [0,0,1,1]$. It has $q^2$ neighboring directed chambers with respect to $L_I$, namely $\{g_j IZ\}_{j \in \J''}$.
 The only neighbor of $IZ$ contained in $\Sigma_0$ (cf. Figure 1) is the chamber $t IZ$, formed by vertices $[0,0,1,1]$, $[0,0,1,2]$, and $[-1,0,0,1]$ with the directed edge $[0,0,1,1] \to [-1,0,0,1]$. In general, given  a directed chamber $gIZ$ with directed edge $gKZ \to g'KZ$, its neighboring directed chambers defined through $L_I$ can be  obtained as follows. First take a chamber $g_1 I Z\neq gIZ$, sharing the same type 2 edge containing $g'KZ$ with $gIZ$. Then any chamber $g_2 IZ \neq g_1 IZ$ that contains the type 1 edge of $g_1 IZ$ with the directed edge from $g'KZ$ to the other special vertex in $g_2 IZ$ is a neighboring chamber of $gIZ$. As each edge is contained in exactly $q+1$ chambers, it is easy to check that exactly $q^2$ directed chambers can be obtained
 this way.

\subsection{Adjacency operators defined on $X_{\Gamma}$}

The vertex adjacency operators $A_1$, $A_2$ defined on the space $C(G/KZ)$ naturally induce operators on the
finite-dimensional space $C(\Gamma \bb G/KZ)$, also denoted by $A_1$, $A_2$ when no ambiguity arises.
So for $f \in C(\Gamma \bb G/KZ)$, we have
$$A_1 f(\Gamma gKZ)=\sum_{g_iK\in K\diag(1,1,\pi,\pi)K/K} f(\Gamma gg_i KZ) $$ and
$$A_2 f(\Gamma gKZ)=(q^2+1)f(\Gamma gKZ)+ \sum_{h_jK\in K\diag(1, \pi, \pi, \pi^2)K/K} f(\Gamma gh_jKZ).$$

Similarly, we have the edge adjacency operators $L_{P_1}$ and $L_{P_2}$ defined on $C(\Gamma \bb G/{P_1}Z)$
and $C(\Gamma \bb G/{P_2}Z)$ respectively, and
the chamber adjacency operator $L_I$ defined on $C(\Gamma \bb G/IZ)$.
These adjacency operators associated to $X_{\Gamma}$ play major roles in the closed form descriptions of various zeta functions attached to $X_{\Gamma}$, to be discussed in the next section.

\section{Zeta functions of complexes from $\Sp_4(F)$}

Recall that a cycle in a topological space has a starting point and orientation, and repetition of vertices is allowed. It is primitive if it can not be obtained by repeating another cycle more than once.
Two cycles are said to be equivalent if one can be obtained from the other by shifting its starting vertex.

The building $X$ is naturally equipped with a metric coming from the Euclidean metric on each apartment. Geodesics in $X$ are the straight lines in apartments
(cf. Brown \cite{Br}).
We say that a cycle in $X_{\Gamma}$ is geodesic if it can be lifted to a geodesic path in $X$. A geodesic cycle in $X_{\Gamma}$ is called tailless if it
remains geodesic after shifting the starting vertex. All paths considered in this paper are contained in the $1$-skeleton of $X$ or $X_\Gamma$.

\subsection{The zeta function of $X_{\Gamma}$}

The zeta function of $X_{\Gamma}$ is defined as $$Z(X_{\Gamma}, u)=\prod_{[C]} (1-u^{l[C]})^{-1},$$ where $[C]$ runs through the equivalence classes of tailless primitive closed geodesics in $X_{\Gamma}$, and $l([C])$ is the length of any geodesic in $[C]$. Since the geodesics considered above are in the $1$-skeleton of $X_{\Gamma}$, each of them comprises edges of the same type.

Like the zeta functions attached to finite graphs \cite{Ha} and finite complexes arising from the building of $\SL_3(F)$
 \cite{KL}, the zeta function $Z(X_{\Gamma}, u)$ can be expressed in terms of the edge adjacency operators $L_{P_1}$ on $C(\Gamma \bb G/{P_1}Z)$ and $L_{P_2}$ on $C(\Gamma\bb G/{P_2}Z)$:

\begin{theorem}\label{zeta1}
The zeta function $Z(X_{\Gamma}, u)$ is a rational function, given in closed form by
$$ Z(X_{\Gamma}, u)=\frac{1}{\det(I-L_{P_1} u)\det(I-L_{P_2}u^2)}. $$
\end{theorem}

\begin{proof}
For $i = 1, 2$ let $Z_i(X_{\Gamma}, u)=\prod_{[C]} (1-u^{l[C]})^{-1}$, where $[C]$ runs through the equivalence classes of tailless
primitive closed geodesics consisting of only type $i$ edges. An argument similar to the proof of Theorem 2 \cite{ST} gives $u\frac{d}{du}\log Z_i(X_{\Gamma}, u) = \sum_C u^{l(C)}$, where $C$ runs through all tailless closed geodesics consisting of only type $i$ edges. In general, a cycle in $X_{\Gamma}$ given by a sequence of adjacent vertices $ v_0 \to v_1 \to v_2 \dots \to v_n=v_0$ can also be described
in terms of a sequence of directed edges $e_1 \to e_2 \dots \to e_n \to e_1$, where $e_i$ is the edge $v_{i-1} \to v_i$. If $A_e$ denotes the edge adjacency matrix on directed edges in the $1$-skeleton of $X_{\Gamma}$, then $\Tr A_e^n$ counts the number of tailless cycles of length n in $X_{\Gamma}$, as pointed out in \cite{Ha}.

First we claim that
\begin{equation}\label{e1}
Z_1(X_{\Gamma}, u)=\frac{1}{\det(I-L_{P_1} u)}.
\end{equation}
In view of the definition of $L_{P_1}$, $\Tr L_{P_1}^n$ counts precisely the number of type $1$ tailless closed
geodesics of length $n$ in $X_{\Gamma}$. Hence we have
\begin{align*}
u\frac{d}{du}\log Z_1(X_{\Gamma}, u)=\sum_{n \ge 1} \Tr L_{P_1}^n u^n
= u\frac{d}{du}\log \frac{1}{\det(I-L_{P_1} u)}.
\end{align*}
The standard procedure of exponentiating both sides and comparing the constants verifies the claim.

Next consider  $Z_2(X_{\Gamma},u)$. Along a tailless geodesic closed cycle expressed by a sequence of directed type $2$ edges, only every other edges are from special vertices to non-special ones, and in the subsequence of such edges, the consecutive ones are adjacent precisely as described by $L_{P_2}$.
Therefore $\Tr L_{P_2}^n$ counts one half the number of type 2 tailless closed geodesics of length $2n$ in $X_{\Gamma}$, and we have
$$ u\frac{d}{du}\log Z_2(X_{\Gamma}, u)=2 \sum_{n \ge 1} \Tr L_{P_2}^n u^{2n}
                                       =u\frac{d}{du}\log \frac{1}{\det(I-L_{P_2} u^2)},$$
which yields the identity
\begin{equation}\label{e2}
Z_2(X_{\Gamma}, u)=\frac{1}{\det(I-L_{P_2} u^2)}.
\end{equation}

As $Z(X_{\Gamma}, u)=Z_1(X_{\Gamma}, u)Z_2(X_{\Gamma}, u)$, the theorem follows from (\ref{e1}) and (\ref{e2}).

\end{proof}

\subsection{Another closed form expression of the zeta function $Z(X_{\Gamma}, u)$}

Similar to zeta functions of graphs \cite{Ih} and zeta functions of complexes arising from finite quotients of the building attached to $\SL_3(F)$ \cite{KL, KLW}, $Z(X_{\Gamma}, u)$ has another expression which gives topological information and spectral information of the complex $X_\Gamma$.

\begin{theorem}\label{zeta2}
$Z(X_{\Gamma}, u)$ has another closed form expression
in terms of the vertex adjacency operators $A_1$ and $A_2$ on $C(\Gamma \bb G/KZ)$ and chamber adjacency operator $L_I$ on $C(\Gamma \bb G/IZ)$:
$$Z(X_{\Gamma}, u)=\frac{(1-u^2)^{\chi(X_{\Gamma})}(1-q^2 u^2)^{-(q^2-1)N_p}}{\det(I-A_1 u+qA_2 u^2-q^3 A_1 u^3+q^6I u^4)\det(I-L_I u)},$$
where $\chi(X_{\Gamma})$ is the Euler characteristics of $X_{\Gamma}$, and $N_p$ is the number of primitive vertices in $X_{\Gamma}$.
\end{theorem}

Combining Theorems \ref{zeta1} and \ref{zeta2}, we obtain the following identity on the operators on $X_\Gamma$ defined before.

\begin{cor}\label{zetaid} With the same notation as in Theorems  \ref{zeta1} and \ref{zeta2}, we have
\begin{equation}\label{id}
\frac{(1-u^2)^{\chi(X_{\Gamma})}(1-q^2 u^2)^{-(q^2-1)N_p}}{\det(I-A_1 u+qA_2 u^2-q^3 A_1 u^3+q^6I u^4)}
= \frac{\det(I-L_I u)}{\det(I-L_{P_1} u)\det(I-L_{P_2}u^2)}.
\end{equation}
\end{cor}

To prove Theorem \ref{zeta2}, it suffices to establish the identity (\ref{id}). We shall follow the approach in \cite{KLW} by comparing the eigenvalues
of each operator, outlined below.

The operators $A_1$ and $A_2$ act on $C(\Gamma \bb G/KZ)$. Since the space $C(\Gamma \bb G/KZ)$ is finite dimensional, it can be regarded as the $K$-invariant subspace
$L^2(\Gamma \bb G/Z)^K$ in $L^2(\Gamma \bb G/Z)$. Similarly, we have $C(\Gamma \bb G/P_iZ)=L^2(\Gamma \bb G/Z)^{P_i}$ for $i=1,2$, and $C(\Gamma \bb G/IZ)=L^2(\Gamma \bb G/Z)^I$. The group $G$ acts on $L^2(\Gamma \bb G/Z)$ by right translations. As $\Gamma$ is discrete and cocompact-mod-center, the representation space $L^2(\Gamma \bb G/Z)$ is
decomposed into a direct sum of irreducible unitary representations of $G$ with trivial central character. To understand the actions of $A_1$, $A_2$, $L_{P_1}$, $L_{P_2}$, and $L_I$ for the sake of proving
the identity (\ref{id}), it then suffices to compute the actions of these operators on each irreducible representation $(\rho, V)$ with non-trivial $I$-invariant vectors, namely the Iwahori-spherical representation, that occurs in the decomposition of $L^2(\Gamma \bb G/Z)$. The unitary representations of $G$ are classified by Sally-Tadi\'c \cite{SaT}; while the irreducible Iwahori-spherical representations are determined by Borel \cite{Bo} and Casselman \cite{Ca} as constituents of unramified principal series representations of $G$. Therefore, we can realize each $(\rho, V)$ in the standard induced model to carry out computations.
Explicit computations will be dealt with in \S5, and a complete proof of the identity (\ref{id}) will be given in \S6.

\subsection{Representation-theoretic background}
The following tables are adapted from Schmidt \cite{Sc1}. Table 1 gives a list of all irreducible unitary Iwahori-spherical representations of $G$, which are divided into six types. Table 2 records the dimensions of the spaces of parahori-invariant vectors for each type of representations classified in Table 1. Here the parahoric subgroups $K$, $P_{02}$, $P_2$, $P_1$, and $I$ are defined as in $\S$ 2.3. For each type of representation $(\rho, V)$, the two columns $C_1$ and $C_2$ record the numbers $(2 \dim V^K + \dim V^{P_{02}}) -(\dim V^{P_1} + 2 \dim V^{P_2}) + \dim V^I$ and $4 \dim V^K -(\dim V^{P_1} + 2 \dim V^{P_2}) + \dim V^I$, respectively; these data will be needed in the proof of Proposition \ref{multi}. The last column of Table 2 indicates tempered representations, which are related to Ramanujan complexes discussed in \S 7.

The characters $\chi_1$, $\chi_2$, $\chi$, $\sigma$ of $F^\times$ in the tables are all assumed unramified. The notation $\nu$ denotes the absolute value of $F$, and the notation $\xi$ denotes the nontrivial unramified quadratic character of $F^\times$. For a character $\chi$ of $F^\times$, $e(\chi)$ is defined by $|\chi(x)|=|x|^{e(\chi)}$.

A type I representation is obtained by normalized parabolic induction from a character of the Borel subgroup $B$.
It is denoted by $\chi_1 \times \chi_2 \rtimes \sigma$ if the corresponding character of $B$ is given by
$$\left(\begin{matrix}a&*&*&*\\0&b&*&*\\0&0&cb^{-1}&*\\0&0&0&ca^{-1}\end{matrix}\right) \to \chi_1(a)\chi_2(b)\sigma (c).$$ Such a representation is irreducible if and only if $\chi_1 \ne \nu^{\pm 1}, \chi_2 \ne \nu^{\pm 1}$ and $\chi_1 \ne \nu^{\pm 1}\chi_2^{\pm 1}$.
Representations of types II and III are the constituents of $\nu^{1/2}\chi \times \nu^{-1/2}\chi \rtimes \sigma$ with $\chi \notin \{\nu^{\pm 1}, \nu^{\pm 3}\}$, and the constituents of $\chi \times \nu \rtimes \nu^{-1/2}\sigma$ with $\chi \notin \{\triv_{F^\times} , \nu^{\pm 2}\}$,
respectively. The type IV representations are further divided into four sub-types, corresponding to the four constituents of $\nu^2 \times \nu \rtimes \nu^{-3/2}\sigma$. We only include types IVa and IVd in the tables since the other two sub-types are never unitary. Representations of types V and VI are the constituents of $\nu \xi \times \xi \rtimes \nu^{-1/2}\sigma$ and $\nu \times \triv_{F^\times} \rtimes \nu^{-1/2}\sigma$, respectively. All the representations except for those of type IVd are infinite-dimensional. The readers are referred to Schmidt \cite[\S 1.1--\S 1.3]{Sc1} for further discussion of each type of representations.

\medskip
\noindent

\begin{tabular}{|c|c|c|} \hline
  \text{type}&\text{representation}&\text{conditions for unitarity}\\\hline

  &&$e(\chi_1)=e(\chi_2)=e(\sigma)=0$\\\cline{3-3}
  &&$\chi_1=\nu^\beta\chi,\:\chi_2=\nu^\beta\chi^{-1},\:e(\sigma)=-\beta$,\\
  &&$e(\chi)=0,\:\chi^2\neq1,\:0<\beta<1/2$\\\cline{3-3}
  \raisebox{0.5ex}[-0.5ex]{\rm I}&\raisebox{0.5ex}[-0.5ex]{$\chi_1\times\chi_2\rtimes\sigma
   \quad\text{(irreducible)}$}
   &$\chi_1=\nu^\beta,\:e(\chi_2)=0,\:e(\sigma)=-\beta/2$,\\
   &&$\chi_2\neq1,\:0<\beta<1$\\\cline{3-3}
  &&$\chi_1=\nu^{\beta_1}\chi,\:\chi_2=\nu^{\beta_2}\chi,\:
   e(\sigma)=(-\beta_1-\beta_2)/2$,\\
  &&$\chi^2=1,\:0\leq\beta_2\leq\beta_1,\:0<\beta_1<1,\:\beta_1+\beta_2<1$\\\hline

  &&$e(\sigma)=e(\chi)=0$\\\cline{3-3}
  \raisebox{2.5ex}[-2.5ex]{\rm IIa}&\raisebox{2.5ex}[-2.5ex]{$\chi\St_{\GL_2}\rtimes\sigma$}
   &$\chi=\xi\nu^\beta,\:e(\sigma)=-\beta,\:\xi^2=1,\:0<\beta<1/2$\\ \cline{2-3}
  &&$e(\sigma)=e(\chi)=0$\\\cline{3-3}
  \raisebox{2.5ex}[-2.5ex]{\rm IIb}&\raisebox{2.5ex}[-2.5ex]{$\chi\triv_{\GL_2}\rtimes\sigma$}
   &$\chi=\xi\nu^\beta,\:e(\sigma)=-\beta,\:\xi^2=1,\:0<\beta<1/2$\\\hline

  {\rm IIIa}&$\chi\rtimes\sigma\St_{\GSp_2}$&$e(\sigma)=e(\chi)=0$\\\cline{2-3}
  {\rm IIIb}&$\chi\rtimes\sigma\triv_{\GSp_2}$&$e(\sigma)=e(\chi)=0$\\\hline

  {\rm IVa}&$\sigma\St_{\GSp_4}$&$e(\sigma)=0$\\\cline{2-3}
  {\rm IVd}&$\sigma\triv_{\GSp_4}$&$e(\sigma)=0$\\\hline

  {\rm Va}&$\delta([\xi,\nu\xi],\nu^{-1/2}\sigma)$&$e(\sigma)=0$\\\cline{2-3}
  {\rm Vb}&$L(\nu^{1/2}\xi\St_{\GL_2},\nu^{-1/2}\sigma)$&$e(\sigma)=0$\\\cline{2-3}
  {\rm Vc}&$L(\nu^{1/2}\xi\St_{\GL_2},\xi\nu^{-1/2}\sigma)$&$e(\sigma)=0$\\\cline{2-3}
  {\rm Vd}&$L(\nu\xi,\xi\rtimes\nu^{-1/2}\sigma)$&$e(\sigma)=0$\\\hline

  {\rm VIa}&$\tau(S,\nu^{-1/2}\sigma)$&$e(\sigma)=0$\\\cline{2-3}
  {\rm VIb}&$\tau(T,\nu^{-1/2}\sigma)$&$e(\sigma)=0$\\\cline{2-3}
  {\rm VIc}&$L(\nu^{1/2}\St_{\GL_2},\nu^{-1/2}\sigma)$&$e(\sigma)=0$\\\cline{2-3}
  {\rm VId}&$L(\nu,1_{F^\times}\rtimes\nu^{-1/2}\sigma)$&$e(\sigma)=0$\\\hline

\end{tabular}

\begin{center}
Table 1
\end{center}

\medskip

\begin{tabular}{|c|c|c|c|c|c|c|c|c|c|c|} \hline
\text{type} &\text{representation } $(\rho, V)$ &$V^K$ &$V^{P_{02}}$ &$V^{P_2}$ &$V^{P_1}$ &$V^I$&&$C_1$&$C_2$&\text{tempered} \\
\hline

   {\rm I}&$\chi_1\times\chi_2\rtimes\sigma$
   &1&2&4&4&8  &&0&0  &$\chi_i$, $\sigma$ unitary \\
   \hline

   {\rm IIa}&$\chi\St_{\GL_2}\rtimes\sigma$&0&1&2&1&4   &&0&-1   &$\chi$, $\sigma$ unitary \\
   {\rm IIb}&$\chi\triv_{\GL_2}\rtimes\sigma$&1&1&2&3&4   &&0&1  & \\
   \hline

   {\rm IIIa}&$\chi\rtimes\sigma\St_{\GSp_2}$&0&0&1&2&4   &&0&0  &$\chi$, $\sigma$ unitary \\
   {\rm IIIb}&$\chi\rtimes\sigma\triv_{\GSp_2}$&1&2&3&2&4   &&0&0  &\\
   \hline

   {\rm IVa}&$\sigma\St_{\GSp_4}$&0&0&0&0&1     &&1&1   &$\sigma$ unitary \\
   {\rm IVd}&$\sigma\triv_{\GSp_4}$&1&1&1&1&1     &&1&2   &\\
   \hline

   {\rm Va}&$\delta([\xi,\nu\xi],\nu^{-\frac{1}{2}}\sigma)$&0&0&1&0&2     &&0&0    &$\sigma$ unitary \\
   {\rm Vb}&$L(\nu^{\frac{1}{2}}\xi\St_{\GL_2},\nu^{-\frac{1}{2}}\sigma)$&0&1&1&1&2   &&0&-1  &\\
   {\rm Vc}&$L(\nu^{\frac{1}{2}}\xi\St_{\GL_2},\xi\nu^{-\frac{1}{2}}\sigma)$&0&1&1&1&2    &&0&-1   &\\
   {\rm Vd}&$L(\nu\xi,\xi\rtimes\nu^{-\frac{1}{2}}\sigma)$&1&0&1&2&2    &&0&2   &\\
   \hline

   {\rm VIa}&$\tau(S,\nu^{-\frac{1}{2}}\sigma)$&0&0&1&1&3    &&0&0    &$\sigma$ unitary\\
   {\rm VIb}&$\tau(T,\nu^{-\frac{1}{2}}\sigma)$&0&0&0&1&1    &&0&0    &$\sigma$ unitary  \\
   {\rm VIc}&$L(\nu^{\frac{1}{2}}\St_{\GL_2},\nu^{-\frac{1}{2}}\sigma)$&0&1&1&0&1    &&0&-1   &\\
   {\rm VId}&$L(\nu,\triv_{F^\times}\rtimes\nu^{-\frac{1}{2}}\sigma)$&1&1&2&2&3   &&0&1   &\\
   \hline

\end{tabular}

\begin{center}
Table 2
\end{center}

\smallskip

We shall be dealing with irreducible unitary Iwahori-spherical representations of $G$ with trivial central characters. The condition on central characters
amounts to requiring that $\chi_1 \chi_2 \sigma^2=\triv_{F^\times}$ for type I representations; $\chi^2 \sigma^2=\triv_{F^\times}$ for type II representations;
$\chi \sigma^2 =\triv_{F^\times}$ for type III representations; and $\sigma ^2=\triv_{F^\times}$ for representations of types IV, V, and VI.

Let $m_i$ denote the number of type $i$ representations, counting multiplicity, in $L^2(\Gamma \bb G/Z)$. For instance,
there are two type IVd representations, each occurring once in $L^2(\Gamma \bb G/Z)$. These are both one-dimensional representations given by the
trivial character and its unramified quadratic twist $\xi$. Hence $m_{IVd}=2$. The notation $m_i$ is well-defined except in the cases of
types Vb and Vc, where a representation of one type can be regarded as a representation of the other by replacing $\sigma$ with $\sigma \xi$.
Therefore we shall not distinguish these two types and let $m_{Vbc}$ denote the number of type Vb or type Vc representations, counting multiplicity, in $L^2(\Gamma \bb G/Z)$.

\begin{prop}\label{multi}
(1) The Steinberg representation $\St_{\GSp_4}$ of type IVa occurs with multiplicity $\chi(X_{\Gamma})-1$ in $L^2(\Gamma \bb G/Z)$.

(2) The following identity holds: $$-m_{IIa}+m_{IIb}-m_{Vbc}+2m_{Vd}-m_{VIc}+m_{VId}+2=-2(q^2-1)N_p.$$
\end{prop}

\begin{proof}
(1) Suppose $(\rho, V)$ is an irreducible Iwahori-spherical representation that occurs in the
decomposition of $L^2(\Gamma \bb G/Z)$. As recorded in column $C_1$ of Table 2,
when $\rho$ is not of type IVa or IVd we have
\begin{equation}\label{st1}
(2 \dim V^K + \dim V^{P_{02}}) -(\dim V^{P_1} + 2 \dim V^{P_2}) + \dim V^I = 0;
\end{equation}
otherwise we have
\begin{equation}\label{st2}
(2 \dim V^K + \dim V^{P_{02}}) -(\dim V^{P_1} + 2 \dim V^{P_2}) + \dim V^I = 1.
\end{equation}

Let $N_0$, $N_1$, and $N_2$ denote the numbers of vertices, edges, and chambers in $X_{\Gamma}$ respectively,
then $\chi(X_{\Gamma})=N_0-N_1+N_2$. From the parametrizations of vertices, edges, and chambers in $X_{\Gamma}$,
we have $2N_0=2\dim L^2(\Gamma \bb G/Z)^K + \dim L^2(\Gamma \bb G/Z)^{P_{02}}$, $2N_1=\dim L^2(\Gamma \bb G/Z)^{P_1}+
2\dim L^2(\Gamma \bb G/Z)^{P_2}$, and $2N_2=\dim L^2(\Gamma \bb G/Z)^I.$

Summing up the identities (\ref{st1}) and (\ref{st2}) over all the irreducible Iwahori-spherical representations, with multiplicities,  in
$L^2(\Gamma \bb G/Z)$  yields
$$2N_0-2N_1+2N_2= 2\chi(X_{\Gamma})= 2 + m_{IVa}.$$
Since $L^2(\Gamma \bb G/Z)$ contains two Steinberg representations, $\St_{\GSp_4}$ and $\xi \St_{\GSp_4}$, with the same
multiplicity, the multiplicity of each Steinberg representation equals one half the total multiplicity $m_{IVa}$, which is exactly $\chi(X_{\Gamma}) -1$.

(2) Write $m$ for the left-hand side of the identity to be verified. The same argument as in (1) applied to column $C_2$ yields
\begin{equation}\label{m}
m-2+m_{IVa}+2m_{IVd}=4 \dim L^2(\Gamma \bb G/Z)^K - 2N_1 +2N_2.
\end{equation}
It follows from the discussion above that $m_{IVa}=2N_0-2N_1+2N_2-2$, and $m_{IVd}=2$. Combined with  $$\dim L^2(\Gamma \bb G/Z)^K=2N_p=\frac{2}{q^2+3}N_0$$
from Proposition \ref{vertex}, (\ref{m}) gives rise to the desired expression for $m$.
\end{proof}

\section{Eigenvalue computation}
In this section we compute the eigenvalues of the vertex adjacency operators $A_1$, $A_2$, the edge adjacency operators $L_{P_1}$, $L_{P_2}$, and the
chamber adjacency operator $L_I$ on each type of representations classified in Table 1. Let
$$R(u)= \frac{\det(I-L_I u)\det(I-A_1 u+qA_2 u^2-q^3 A_1 u^3+q^6I u^4)}{\det(I-L_{P_1} u)\det(I-L_{P_2}u^2)}.$$
We shall record the contribution to $R(u)$ from each type of representations.

\subsection{Principal series representations}
We begin with a principal series representation $\rho_0 \simeq \chi_1 \times \chi_2 \rtimes \sigma$, induced from unramified characters $\chi_1$, $\chi_2$, $\sigma$ of
$F^\times$ with $\chi_1 \chi_2 \sigma^2=\triv_{F^\times}$. The representation $\rho_0$ is not assumed to be irreducible. It can be
realized in the space $V$ consisting of functions $f: G \to \mathbb{C}$ satisfying the transformation property
$$f(hg)=|a^2bc^{-\frac{3}{2}}|\chi_1(a)\chi_2(b)\sigma(c) f(g), \quad \text{for all } h=\left(\begin{matrix}a&*&*&*\\0&b&*&*\\0&0&cb^{-1}&*\\0&0&0&ca^{-1}\end{matrix}\right) \in B,$$ with right regular actions by $G$.
Let $$W=\{id, s_1, s_2, s_1s_2, s_2s_1, s_1s_2s_1, s_2s_1s_2, s_1s_2s_1s_2\}$$ be the Weyl group of $G$ with the generators $s_1$, $s_2$ defined
in \S 2.3. It follows from the decomposition $G=\bigsqcup_{w \in W} BwI$ that the space of Iwahori-invariant vectors $V^I$ has dimension eight.
Let $f_{\alpha}(x)$ be the function in $V$ supported on the coset $B\alpha I$ with $f_{\alpha}(\alpha I)=1$. Then $V^I$ has a basis consisting of
$f_1:=f_{id}$, $f_2:=f_{s_1}$, $f_3:=f_{s_2}$, $f_4:=f_{s_1s_2}$, $f_5:=f_{s_2s_1}$, $f_6:=f_{s_1s_2s_1}$, $f_7:=f_{s_2s_1s_2}$, $f_8:=f_{s_1s_2s_1s_2}$.

\subsubsection{Eigenvalues of $L_I$}
The operator $L_I$ acts on $V^I$ as described in \S3. A direct computation using the coset decomposition in (\ref{L_I}) gives the explicit description of $L_I$
on the chosen basis:
\begin{equation}\label{eigenL_I}
\begin{aligned}
L_I f_1&=q^{\f{1}{2}}\chi_2 \sigma(\pi) f_3, \quad
L_I f_2=(q-1)q^{\f{1}{2}}\chi_2 \sigma(\pi) f_3+q^{\f{1}{2}}\chi_1 \sigma(\pi) f_4,   \\
L_I f_3&=q^{\f{3}{2}}\sigma(\pi)f_1, \qquad
L_I f_4=q^{\f{3}{2}}\sigma(\pi)f_2,      \\
L_I f_5&=(q-1)q^{\f{3}{2}}\sigma(\pi)f_1+(q-1)q^{\f{1}{2}}\chi_1\sigma(\pi)f_4+q^{\f{1}{2}}\chi_1\chi_2\sigma(\pi)f_7,     \\
L_I f_6&=(q-1)q^{\f{3}{2}}\sigma(\pi)f_2+(q-1)q^{\f{3}{2}}\chi_2\sigma(\pi)f_3+q^{\f{1}{2}}\chi_1\chi_2\sigma(\pi)f_8,    \\
L_I f_7&=(q-1)q^{\f{3}{2}}\sigma(\pi)f_2+q^{\f{3}{2}}\chi_2\sigma(\pi)f_5,    \\
L_I f_8&=(q-1)q^{\f{5}{2}}\sigma(\pi)f_1+(q-1)^2q^{\f{3}{2}}\sigma(\pi)f_2+(q-1)q^{\f{3}{2}}\chi_2\sigma(\pi)f_5+q^{\f{3}{2}}\chi_1\sigma(\pi)f_6.
\end{aligned}
\end{equation}
\noindent This shows that $L_I$ on $V^I$ has eigenvalues $\pm \sqrt{q^2\chi_1(\pi)}$,  $\pm \sqrt{q^2\chi_2(\pi)}$, $\pm \sqrt{q^2\chi_1^{-1}(\pi)}$, and $\pm \sqrt{q^2\chi_2^{-1}(\pi)}$.

\subsubsection{Eigenvalues of $L_{P_1}$, $L_{P_2}$}
Let $W_i$ be the subgroup of order 2 in $W$ generated by $s_i$, for $i=1,2$.
Since $P_1=I\cup Is_1I$, the group $G=\bigsqcup_{w \in W}BwI$ can be written as $\bigsqcup_{w \in W/W_1} BwP_1$, where $W/W_1$ is represented by the Weyl elements $\{id, s_2, s_1s_2, s_2s_1s_2\}$.
Hence the space of $P_1$-invariant vectors $V^{P_1}$ has dimension four. Let $g_{\alpha}(x)$ be the function in $V$ supported on the coset $B\alpha P_1$
with $g_{\alpha}(\alpha P_1)=1$. Then $V^{P_1}$ has a basis consisting of $g_1:=g_{id}$, $g_2:=g_{s_2}$, $g_3:=g_{s_1s_2}$,
$g_4:=g_{s_2s_1s_2}$.
Applying the coset decomposition in (\ref{L_{P_1}}), the action of $L_{P_1}$ on $V^{P_1}$ can be described as follows.
\begin{equation}\label{eigenL_{P_1}}
\begin{aligned}
L_{P_1}g_1&=q^{\f{3}{2}}\sigma(\pi)g_1,\\
L_{P_1}g_2&=(q-1)q^{\f{3}{2}}\sigma(\pi)g_1+q^{\f{3}{2}}\chi_2\sigma(\pi)g_2,\\
L_{P_1}g_3&=(q-1)q^{\f{5}{2}}\sigma(\pi)g_1+(q-1)q^{\f{3}{2}}\chi_2\sigma(\pi)g_2+q^{\f{3}{2}}\chi_1\sigma(\pi)g_3,\\
L_{P_1}g_4&=(q-1)q^{\f{7}{2}}\sigma(\pi)g_1+(q-1)q^{\f{5}{2}}\chi_2\sigma(\pi)g_2+(q-1)q^{\f{3}{2}}\chi_1\sigma(\pi)g_3+q^{\f{3}{2}}\chi_1\chi_2\sigma(\pi)g_4.
\end{aligned}
\end{equation}
\noindent Hence $L_{P_1}$ on $V^{P_1}$ has eigenvalues $q^{\f{3}{2}}\chi_1\sigma(\pi)$, $q^{\f{3}{2}}\chi_2\sigma(\pi)$, $q^{\f{3}{2}}\chi_1\chi_2\sigma(\pi)$, and $q^{\f{3}{2}}\sigma(\pi)$.

\smallskip

Similarly, we have $P_2=I \cup Is_2I$ and $G=\bigsqcup_{w \in W/W_2} BwP_2$, with the quotient $W/W_2$ represented by $\{id, s_1, s_2s_1, s_1s_2s_1\}$.
Let $h_{\alpha}(x)$ be the function in $V$ supported on $B\alpha P_2$ with $h_{\alpha}(\alpha P_2)=1$.
Then $h_1:=h_{id}$, $h_2:=h_{s_1}$, $h_3:=h_{s_2s_1}$, and $h_4:=h_{s_1s_2s_1}$ form a basis of the 4-dimensional subspace $V^{P_2}$.
Using (\ref{L_{P_2}}), we obtain the action of $L_{P_2}$ on $V^{P_2}$:

\begin{equation}\label{eigenL_{P_2}}
\begin{aligned}
L_{P_2}h_1=&q^2\chi_2\sigma^2(\pi)h_1,\\
L_{P_2}h_2=&(q-1)q^2\chi_2\sigma^2(\pi)h_1+q^2\chi_1\sigma^2(\pi)h_2,\\
L_{P_2}h_3=&(q-1)q^3\chi_2\sigma^2(\pi)h_1+((q-1)q^2\chi_1\sigma^2(\pi)+(q-1)q^2\chi_1\chi_2\sigma^2(\pi))h_2+q^2\chi_1\chi_2^2\sigma^2(\pi)h_3,\\
L_{P_2}h_4=&((q-1)q^3\chi_1\chi_2\sigma^2(\pi)+(q-1)q^4\chi_2\sigma^2(\pi))h_1+((q-1)^2q^2\chi_1\chi_2\sigma^2(\pi)\\
           &+(q-1)q^3\chi_1\sigma^2(\pi))h_2+(q-1)q^2\chi_1\chi_2^2\sigma^2(\pi)h_3+q^2\chi_1^2\chi_2\sigma^2(\pi)h_4.
\end{aligned}
\end{equation}
\noindent Hence $L_{P_2}$ on $V^{P_2}$ has eigenvalues $q^2\chi_1(\pi)$, $q^2\chi_2(\pi)$, $q^2\chi_1^{-1}(\pi)$, and $q^2\chi_2^{-1}(\pi)$.

\subsubsection{Eigenvalues of $A_1$ and $A_2$}

The subspace $V^K$ is one-dimensional, generated by $\sum_{i=1}^8 f_i$. It follows from the coset decompositions in (\ref{A_1}) and (\ref{A_2}) that the Hecke operators $A_1$, $A_2$ act on $V^K$ via the scalar multiplications by $\lambda_1$ and $\lambda_2$ respectively, where
\begin{equation}\label{eigenA_1}
\begin{aligned}
\lambda_1&= q^{\f{3}{2}}(\chi_1\chi_2(\pi)+\chi_1(\pi)+\chi_2(\pi)+1)\sigma(\pi),\\
\lambda_2&= q^2(\chi_1^2\chi_2(\pi)+\chi_1\chi_2^2(\pi)+\chi_1(\pi)+\chi_2(\pi)+2\chi_1\chi_2(\pi))\sigma^2(\pi).
\end{aligned}
\end{equation}

In what follows let $(\rho, V_{\rho})$ be a unitary irreducible Iwahori-spherical representation in $L^2(\Gamma \bb G/Z)$, so it is one of the fifteen sub-types listed in Table 1. We discuss the contribution of $\rho$ to $R(u)$.

\subsection{Type I representations}
 Suppose $\rho$ has type I, then it has the form $\chi_1 \times \chi_2\rtimes \sigma$ with $\chi_1\chi_2\sigma^2=\triv_{F^\times}$. The eigenvalues of $L_I$ on $V_{\rho}^I$, $L_{P_1}$ on $V_{\rho}^{P_1}$, $L_{P_2}$ on $V_{\rho}^{P_2}$, and $A_1$, $A_2$ on $V_{\rho}^K$ are computed in \S5.1. One checks easily from the eigenvalues of $A_1$ and $A_2$ on $V_{\rho}^K$ that the zeros of $\det(I-A_1u+qA_2u^2-q^3A_1u^3+q^6Iu^4)$ on $V_\rho$ are $q^{-\f{3}{2}}\chi_1^{-1}\sigma^{-1}(\pi)$, $q^{-\f{3}{2}}\chi_2^{-1}\sigma^{-1}(\pi)$, $q^{-\f{3}{2}}\chi_1^{-1}\chi_2^{-1}\sigma^{-1}(\pi)$, and $q^{-\f{3}{2}}\sigma^{-1}(\pi)$. Comparing the zeroes of $\det(I-L_Iu)$, $\det(I-L_{P_1}u)$, $\det(I-L_{P_2}u)$, and $\det(I-A_1u+qA_2u^2-q^3A_1u^3+q^6Iu^4)$ on $V_\rho$, we see that the net contribution to $R(u)$ from $\rho$ is 1. Hence the total contribution to $R(u)$ from type I representations is trivial.

\subsection{Type II representations}
Suppose $\rho$ has type II, then it is a constituent of $\nu^{-\frac{1}{2}}\chi \times \nu^{\frac{1}{2}}\chi \rtimes \sigma$ with $\chi^2\sigma^2=\triv_{F^\times}$.
There are two constituents in the representation $\nu^{-\frac{1}{2}}\chi \times \nu^{\frac{1}{2}}\chi \rtimes \sigma$.
\subsubsection{Type IIb}
Consider first the case when $\rho \simeq \chi\triv_{\GL_2}\rtimes\sigma$.
In this case $\rho$ is the unique subrepresentation of $\nu^{-\frac{1}{2}}\chi \times \nu^{\frac{1}{2}}\chi \rtimes \sigma$. It can be realized in the parabolically induced
space $V_{\rho}$ consisting of functions $f: G \to \mathbb{C}$ satisfying
$$f(hg)=|\det(A)^{\f{3}{2}}c^{-\f{3}{2}}|\chi(\det(A))\sigma(c) f(g) \quad \text{for all }
h=\left(\begin{matrix}A&*\\0&cA'\end{matrix}\right) \in P,$$
where $A'=\left(\begin{matrix}0&1\\1&0\end{matrix}\right) ^tA^{-1}\left(\begin{matrix}0&1\\1&0\end{matrix}\right)$ and $P$ is the Siegel parabolic subgroup of $G$. To describe the basis of parahori-invariant subspaces, we shall let $f_{\alpha}$ (resp. $g_{\alpha}$, $h_{\alpha}$) denote the function in $V_{\rho}^I$ (resp. $V_{\rho}^{P_1}$, $V_{\rho}^{P_2}$) supported on the coset $P\alpha I$ (resp. $P\alpha P_1$, $P\alpha P_2$), taking value 1 at $\alpha$.

The space $V_{\rho}^I$ has dimension four. This can be observed from the decomposition $G=\bigsqcup_{w \in W} BwI = \bigsqcup_{w \in W_1\bb W} PwI$,
where the identity follows from the fact that $s_1 \in P$.  The quotient $W_1\bb W$ has four elements represented by $\{id, s_2, s_2s_1, s_2s_1s_2\}$.
Therefore $V_{\rho}^I$ is generated by $f_1:=f_{id}$, $f_2:=f_{s_2}$, $f_3:=f_{s_2s_1}$, $f_4:=f_{s_2s_1s_2}$.

The space $V_{\rho}^{P_1}$ has dimension three, while $V_{\rho}^{P_2}$ has dimension two, which can be seen from the decompositions
$G=\bigsqcup_{w \in W_1 \bb W/W_1} PwP_1$ and $G=\bigsqcup_{w \in W_1 \bb W/W_2} PwP_2$ respectively. Note that $$W_1 \bb W/W_1 =\{id, s_2, s_2s_1s_2\}\quad
\text{ and } \quad W_1 \bb W/W_2=\{id, s_2s_1\}.$$ So $V_{\rho}^{P_1}$ is generated by $g_1:=g_{id}$, $g_2:=g_{s_2}$ and $g_3:=g_{s_2s_1s_2}$; while $V_{\rho}^{P_2}$ is generated by $h_1:=h_{id}$ and  $h_2:=h_{s_2s_1}$. The space $V_{\rho}^K$ is 1-dimensional, generated by $\sum_{i=1}^4 f_i$.

Similar computation as in \S5.1 gives the explicit actions of $L_I$, $L_{P_1}$, $L_{P_2}$ on the given basis:

$L_{I}(f_{1})=q\chi\sigma(\pi)f_{2}$, \qquad $L_{I}(f_{2})=q^{\frac{3}{2}}\sigma(\pi)f_{1}$,

$L_{I}(f_{3})=(q-1)q^{\frac{3}{2}}\sigma(\pi)f_{1}+ (q-1)q\chi\sigma(\pi)f_{2}+ q^{\frac{1}{2}}\chi^{2}\sigma(\pi)f_{4}$,

$L_{I}(f_{4})=(q-1)q^{\frac{5}{2}}\sigma(\pi)f_{1}+q^{2}\chi\sigma(\pi)f_{3}$; \\

$L_{P_{1}}(g_{1})=q^{\frac{3}{2}}\sigma(\pi)g_{1}$,  \qquad    $L_{P_{1}}(g_{2})=(q^{2}-1)q^{\frac{3}{2}}\sigma(\pi)g_{1}+ q^{2}\chi\sigma(\pi)g_{2}$,

$L_{P_{1}}(g_{3})=(q-1)q^{\frac{7}{2}}\sigma(\pi)g_{1}+(q-1)q^{2}\chi\sigma(\pi)g_{2}+q^{\frac{3}{2}}\chi^{2}\sigma(\pi)g_{3}$;\\

$L_{P_{2}}(h_{1})=q^{\frac{5}{2}}\chi\sigma^{2}(\pi)h_{1}$,

$L_{P_{2}}(h_{2})=((q-1)q^{3}\chi^{2}\sigma^{2}(\pi)+(q^{2}-1)q^{\frac{5}{2}}\chi\sigma^{2}(\pi))h_{1}+q^{\frac{5}{2}}\chi^{3}\sigma^{2}(\pi)h_{2}$.

\noindent
Therefore $L_I$ on $V_{\rho}$ has eigenvalues $\pm \sqrt{q^{\frac{5}{2}}\chi(\pi)}$, $\pm \sqrt{q^{\frac{5}{2}}\chi^{-1}(\pi)}$;  $L_{P_{1}}$ on $V_{\rho}$ has eigenvalues $q^{\frac{3}{2}}\sigma(\pi)$,  $q^{\frac{3}{2}}\sigma^{-1}(\pi)$, $q^{2}\chi\sigma(\pi)$; and
$L_{P_{2}}$ on $V_{\rho}$ has eigenvalues $q^{\frac{5}{2}}\chi(\pi)$, $q^{\frac{5}{2}}\chi^{-1}(\pi)$.
Applying (\ref{eigenA_1}) with $\chi_{1}=\nu^{-\frac{1}{2}}\chi$ and $\chi_{2}=\nu^{\frac{1}{2}}\chi$, we see that the zeroes
of $\det(I- A_{1}u+ qA_{2}u^{2}- q^{3}A_{1}u^{3}+ q^{6}Iu^{4})$ arising from $\rho$ are $q^{-\frac{3}{2}}\sigma(\pi)$, $q^{-\frac{3}{2}}\sigma^{-1}(\pi)$, $q^{-2}\chi\sigma(\pi)$,  and $q^{-1}\chi\sigma(\pi)$.

Comparing the zeros on $V_\rho$ from all four determinants shows that the contribution to $R(u)$ from $\rho$ is $1-q\chi\sigma(\pi)u$.
Each $\rho \simeq \chi\triv_{\GL_2}\rtimes\sigma$ is accompanied with a twist $\xi \rho \simeq \chi\triv_{\GL_2} \rtimes \xi \sigma$ in
$L^2(\Gamma \bb G/Z)$ by the unramified quadratic character $\xi$ with $\xi(\pi)=-1.$ Together,
$\rho$ and $\xi\rho$ contribute the factor $(1-q\chi\sigma(\pi)u)(1-q\chi\sigma\xi(\pi)u)=(1-q^2u^2)$ to $R(u)$.
Hence the total contribution to $R(u)$ from type IIb representations is $(1-q^2u^2)^{\f{1}{2}m_{IIb}}$.

\subsubsection{Type IIa}
Suppose now $\rho \simeq \chi \St_{\GL_2}\rtimes \sigma$. Since $\chi \St_{\GL_2}\rtimes \sigma$ and $\chi\triv_{\GL_2}\rtimes\sigma$ are the two constituents of $\nu^{-\f{1}{2}}\chi \times \nu^{\f{1}{2}}\chi \rtimes \sigma$, the eigenvalues of $L_I$ on $V_{\rho}^I$ can be obtained easily by first finding the eigenvalues of $L_I$ on the full induced space $\nu^{-\f{1}{2}}\chi \times \nu^{\f{1}{2}}\chi \rtimes \sigma$, then excluding those eigenvalues coming from the type IIb representation $\chi\triv_{\GL_2}\rtimes\sigma$ studied in \S5.3.1. The first step can be done using the result of \S5.1.1 with $\chi_1=\nu^{-\f{1}{2}}\chi$ and $\chi_2=\nu^{\f{1}{2}}\chi$. From there, we deduce that the eigenvalues of $L_I$ on $V_{\rho}^I$ are
$\pm \sqrt{q^{\frac{3}{2}}\chi(\pi)}$, $\pm \sqrt{q^{\frac{3}{2}}\chi^{-1}(\pi)}$.

Computing in the same manner shows that the only eigenvalue of $L_{P_{1}}$ on $V_{\rho}^{P_1}$ is $q\chi\sigma(\pi)$, and
the eigenvalues of $L_{P_{2}}$ on $V_{\rho}^{P_2}$ are $q^{\frac{3}{2}}\chi(\pi)$, $ q^{\frac{3}{2}}\chi^{-1}(\pi)$.
There is no nontrivial $K$-invariant vector in $V_{\rho}$. Hence the contribution to $R(u)$ from $\rho$ and its twist $\xi \rho$ is $(1-q\chi\sigma(\pi)u)^{-1}(1-q\chi\sigma\xi(\pi)u)^{-1}=(1-q^2u^2)^{-1}$.
The total contribution to $R(u)$ from type IIa representations is $(1-q^2u^2)^{-\f{1}{2}m_{IIa}}$.

\subsection{Type III representations}
Suppose $\rho$ has type III so that it is a constituent of $\chi \times \nu^{-1} \rtimes  \nu^{\f{1}{2}}\sigma$ with $\chi\sigma^2=\triv_{F^\times}$. There
are two constituents in the representation $\chi \times \nu^{-1} \rtimes  \nu^{\f{1}{2}}\sigma$.

\subsubsection{Type IIIb}
Consider first $\rho \simeq \chi \rtimes \sigma\triv_{\GSp_2}$. It is the unique subrepresentation of $\chi \times \nu^{-1} \rtimes  \nu^{\f{1}{2}}\sigma$, and can be realized in the space $V_{\rho}$ consisting of functions $f: G \to \mathbb{C}$ satisfying
$$f(hg)=|t^2(ad-bc)^{-1}|\chi(t)\sigma(ad-bc)f(g) \quad \text{for all }
h=\left(\begin{matrix}t&*&*&*\\0&a&b&*\\0&c&d&*\\0&0&0&\Delta t^{-1}\end{matrix}\right) \in Q,$$
where $\Delta=ad-bc$, and $Q$ is the Klingen parabolic subgroup of $G$.
Let $f_{\alpha}$, $g_{\alpha}$, $h_{\alpha}$ be defined as in \S 5.3.1 type IIb case, with the Siegel parabolic $P$ there replaced with the
Klingen parabolic $Q$.

Since $s_2\in Q$, the space of Iwahori invariant vectors $V_{\rho}^{I}$ has dimension four from $G= \bigsqcup_{w\in W}BwI= \bigsqcup_{w \in W_2 \bb W}QwI$,
where $W_2\bb W= \{id, s_{1}, s_{1}s_{2}, s_{1}s_{2}s_{1}\}$. Then $V^{I}$ has a basis given by
$f_{1}:= f_{id}, f_{2}:= f_{s_{1}}, f_{3}:= f_{s_{1}s_{2}}$ and $ f_{4}:= f_{s_{1}s_{2}s_{1}}$.

Similar argument shows that $V_{\rho}^{P_{1}}$ and $V_{\rho}^{P_{2}}$ have dimensions two and three respectively.
The space $V_{\rho}^{P_1}$ has a basis consisting of $g_{1}:= g_{id}$ and $ g_{2}:= g_{s_{1}s_{2}}$; while $V_{\rho}^{P_{2}}$ has a basis
consisting of $h_{1}:= h_{id}, h_{2}:= h_{s_{1}}$ and $ h_{3}=h_{s_{1}s_{2}s_{1}}$. The 1-dimensional space $V_{\rho}^K$ is generated by $\sum_{i=1}^4 f_i$.

With the basis given above, the actions of $L_{I}, L_{P_1}, L_{P_2}$ can be described as follows.

$L_{I}(f_{1})=q\sigma(\pi)f_{1}$,

$L_{I}(f_{2})=(q-1)q\sigma(\pi)f_{1}+q\chi\sigma(\pi)f_{3}$,

$L_{I}(f_{3})=q^{2}\sigma(\pi)f_{2}$,

$L_{I}(f_{4})=(q-1)q^{2}\sigma(\pi)f_{1}+(q-1)q^{2}\sigma(\pi)f_{2}+q\chi\sigma(\pi)f_{4}$;\\

$L_{P_{1}}(g_{1})=q^{2}\sigma(\pi)g_{1}$,

$L_{P_{1}}(g_{2})=(q^{2}-1)q^{2}\sigma(\pi)g_{1}+q^{2}\chi\sigma(\pi)g_{2}$; \\

$L_{P_{2}}(h_{1})=q^{2}\sigma^{2}(\pi)h_{1}$,

$L_{P_{2}}(h_{2})=(q^{2}-1)q^{2}\sigma^{2}(\pi)h_{1}+q^{3}\chi\sigma^{2}(\pi)h_{2}$,

$L_{P_{2}}(h_{3})=((q-1)q^{3}\chi\sigma^{2}(\pi)+(q-1)q^{4}\sigma^{2}(\pi))h_{1}+(q-1)q^{3}\chi\sigma^{2}(\pi)h_{2}+q^{2}\chi^{2}\sigma^{2}(\pi)h_{3}.$

Hence the eigenvalues of $L_{I}$ on $V_{\rho}$ are $q\sigma(\pi)$, $q\sigma^{-1}(\pi)$, and $\pm q^{\frac{3}{2}}$.
The eigenvalues of $L_{P_1}$ on $V_{\rho}$ are $q^{2}\sigma(\pi)$ and $q^{2}\sigma^{-1}(\pi)$; and the eigenvalues of $L_{P_2}$ on $V_{\rho}$ are $q^{2}\sigma^2(\pi)$, $q^{2}\sigma^{-2}(\pi)$ and $q^{3}$.
It follows from (\ref{eigenA_1}) that the zeros of
$\det(I- A_{1}u+ qA_{2}u^{2}- q^{3}A_{1}u^{3}+ q^{6}Iu^{4})$ from $\rho$ are $q^{-2}\sigma(\pi)$, $q^{-1}\sigma(\pi)$, $q^{-2}\sigma^{-1}(\pi),$ and $ q^{-1}\sigma^{-1}(\pi)$.
So the contribution to $R(u)$ from $\rho$ is $$\f{(1-q\sigma(\pi)u)(1-q\sigma^{-1}(\pi)u)}{(1+q\sigma(\pi)u)(1+q\sigma^{-1}(\pi)u)}.$$
The representation $\rho$ and its twist $\xi\rho \simeq \chi \rtimes \xi\sigma \triv_{\GSp_2}$ together contribute $1$ to $R(u)$. Hence the total
contribution to $R(u)$ from type IIIb representations is trivial.

\subsubsection {Type IIIa}
Suppose $\rho \simeq \chi\rtimes\sigma\St_{\GSp_2}$, which is the unique quotient of the representation $\chi \times \nu^{-1} \rtimes
\nu^{\f{1}{2}}\sigma$. Following the same argument as in type IIa and applying results from \S5.4.1 and \S5.1 with $\chi_1=\chi$ and $\chi_2=\nu^{-1}$, we deduce that the eigenvalues of $L_I$ on $V_{\rho}^I$ are $-q\sigma(\pi)$, $-q\sigma^{-1}(\pi)$, $\pm q^{\frac{1}{2}}$. The eigenvalues of $L_{P_{1}}$ on $V_{\rho}^{P_1}$ are $q\sigma(\pi)$ and $q\sigma^{-1}(\pi)$; and the only eigenvalue of $L_{P_{2}}$ is $q$. There is no nontrivial $K$-invariant vector in $V_{\rho}$. Hence the contribution to $R(u)$ from $\rho$ is $$\f{(1+q\sigma(\pi)u)(1+q\sigma^{-1}(\pi)u)}{(1-q\sigma(\pi)u)(1-q\sigma^{-1}(\pi)u)}.$$ Pairing $\rho$ with its twist $\xi \rho$, we see that the total contribution to $R(u)$ from type IIIa representations is also trivial.

\subsection{Type IV representations}
Suppose $\rho$ is of type IV, then it is a constituent of $\nu^{-2}\times \nu^{-1}\rtimes \nu^{\f{3}{2}}\sigma$. There are four
constituents in $\nu^{-2}\times \nu^{-1}\rtimes \nu^{\f{3}{2}}\sigma$ corresponding to representations of types IVa-IVd. We shall not concern representations of types IVb and IVc as they will never occur in $L^2(\Gamma \bb G/Z)$.

\subsubsection {Type IVd}

In this case $\rho \simeq \sigma\triv_{\GSp_4}$ with $\sigma^2=\triv_{F^\times}$, which is a one-dimensional representation, and can be regarded as a subrepresentation of $\nu^{-2}\times \nu^{-1}\rtimes \nu^{\f{3}{2}}\sigma$.
Following the notation in \S 5.1 and using the results there with $\chi_1$, $\chi_2$, $\sigma$ in the inducing data equal to $\nu^{-2}$, $\nu^{-1}$ and $\nu^{\f{3}{2}}\sigma$ respectively, it is easy to observe that $V_{\rho}$ is generated by the single function $\sum_{i=1}^8 f_i$, and the eigenvalues of $L_I$, $L_{P_1}$,
$L_{P_2}$ restricted on $V_{\rho}$ are $q^2\sigma(\pi), q^3\sigma(\pi)$, and $q^4$ respectively.
The zeros of $\det(I- A_{1}u+ qA_{2}u^{2}- q^{3}A_{1}u^{3}+ q^{6}Iu^{4})$ from $\rho$ are $q^{-3}\sigma^{-1}(\pi), q^{-2}\sigma^{-1}(\pi), q^{-1}\sigma^{-1}(\pi),$ and $ \sigma^{-1}(\pi)$.

There are two type IVd representations in $L^2(\Gamma \bb G/Z)$; one has $\sigma(\pi)=1$, while the other has $\sigma(\pi)=-1$.
The total contribution to $R(u)$ from type IVd representations is $(1-u^2)(1-q^2u^2)$.

\subsubsection {Type IVa}
Suppose $\rho \simeq \sigma\St_{\GSp_4}$ with $\sigma^2=\triv_{F^\times}$. It can be regarded as a subrepresentation of $\nu^2\times \nu \rtimes \nu^{-\f{3}{2}}\sigma$. As in \S 5.1, we realize the principal series representation $\nu^2\times \nu \rtimes \nu^{-\f{3}{2}}\sigma$ in the standard model $V$ and use the same notation $f_{\alpha}$ for the function in $V$ supported on $B\alpha I$ with $f(\alpha I)=1$. According to Table 2, the spaces $V_{\rho}^{P_1}$, $V_{\rho}^{P_2}$, $V_{\rho}^K$ are all trivial; the space
$V_{\rho}^I$ has dimension one.

To determine $V_{\rho}^I$, consider the intertwining maps $$T_{s_1}: \nu^2 \times \nu \rtimes \nu^{\f{3}{2}}\sigma \to \nu \times \nu^2 \rtimes \nu^{-\f{3}{2}}\sigma$$ and $$T_{s_2}: \nu^2 \times \nu \rtimes \nu^{\f{3}{2}}\sigma \to \nu^2 \times \nu^{-1} \rtimes \nu^{-\f{1}{2}}\sigma$$ defined as in
Casselman $\S 3$ of \cite{Ca}. It follows from Theorem 3.4 in \cite{Ca} that if $l(\alpha w)>l(w)$, then
\begin{equation}\label{Ca}
\begin{aligned}
T_{\alpha}(f_w)&=\f{1}{q}f_w+\f{1}{q}f_{\alpha w} \\
T_{\alpha}(f_{\alpha w})&=f_w+f_{\alpha w},
\end{aligned}
\end{equation}
where $\alpha=s_1$ or $s_2$, $w \in W$, and $l(w)$ denotes the length of $w$. So the maps $T_{s_1}$ and $T_{s_2}$ can be described
explicitly on the basis $\{f_1, \cdots, f_8\}$ of $V^I$. Since $\rho$ is irreducible, and does not appear as a subrepresentation of
$\nu \times \nu^2\rtimes \nu^{-\f{3}{2}}\sigma$ or $\nu^2 \times \nu^{-1} \rtimes \nu^{-\f{1}{2}}\sigma$, the maps $T_{s_1}$ and
$T_{s_2}$ must be trivial on $V_{\rho}$. In particular, if $\phi \in V_{\rho}^I$, then $T_{s_1}(\phi)=T_{s_2}(\phi)=0$.
Writing $\phi$ as a linear combination of $f_1, \cdots, f_8$ and applying (\ref{Ca}) to solve the equations $T_{s_1}(\phi)=0$
and $T_{s_2}(\phi)=0$ simultaneously, we see that $\phi$ is a constant multiple of the function
$$q^4f_1-q^3f_2-q^3f_3+q^2f_4+q^2f_5-qf_6-qf_7+f_8.$$ Hence the eigenvalue of $L_I$ on $V_{\rho}^I$ is $-\sigma(\pi)$,
and the contribution to $R(u)$ from $\rho$ is $1+\sigma(\pi)u$. The total contribution to $R(u)$ from type IVa representations is
$(1-u^2)^{\f{1}{2}m_{IVa}}=(1-u^2)^{\chi(X_{\Gamma})-1}$ by Proposition \ref{multi} (1).

\subsection{Type V representations}
Representations of type V are the constituents of $\nu \xi \times\xi \rtimes \nu^{-1/2}\sigma$. There are four constituents in $\nu \xi \times\xi \rtimes \nu^{-1/2}\sigma$ corresponding to types Va-Vd.

\subsubsection{Types Vb and Vc}
We consider types Vb and Vc together since $L(\nu^{1/2}\xi\St_{\GL_2},\nu^{-1/2}\sigma)$ and $L(\nu^{1/2}\xi\St_{\GL_2},\xi \nu^{-1/2}\sigma)$ are twists
of each other. Suppose $\rho \simeq L(\nu^{1/2}\xi\St_{\GL_2},\nu^{-1/2}\sigma)$ with $\sigma^2=\triv_{F^\times}$, then $\rho$ can be regarded as a subrepresentation in $\nu^{\f{1}{2}}\xi\triv_{\GL_2} \rtimes \nu^{-\f{1}{2}}\xi\sigma$. Since $\nu^{\f{1}{2}}\xi\triv_{\GL_2} \rtimes \nu^{-\f{1}{2}}\xi\sigma$ is parabolically induced from a character on $P$, we realize it in the standard induced model $V$ as usual, and identify $V_{\rho}$ as a subspace in $V$. A basis
for $V_{\rho}^I$ has been given explicitly in Schmidt \cite[Corollary 2.6 (ii)]{Sc3}. It is obtained by exhibiting a nonzero intertwining operator $$A(-\f{1}{2}): \nu^{\f{1}{2}}\xi\triv_{\GL_2} \rtimes \nu^{-\f{1}{2}}\xi\sigma  \to \nu^{-\f{1}{2}}\xi\triv_{GL_2} \rtimes \nu^{\f{1}{2}}\xi\sigma$$ and then studying its restriction on the space of Iwahori-invariant vectors $V^I$ (cf. \cite[Proposition 2.5 (ii)]{Sc3}). Since $\rho$ does not appear as a subrepresentation in $\nu^{-\f{1}{2}}\xi\triv_{GL_2} \rtimes \nu^{\f{1}{2}}\xi\sigma$, the kernal of $A(-\f{1}{2})$ restricted on $V^I$ characterizes $V_{\rho}^I$.

We now record the relevant results from \cite{Sc3}. Let $f_{\alpha}(x)$ be the function in $V$ supported on $P\alpha I$ with $f_{\alpha}(\alpha I)=1$. Then the functions $f_1:=f_{id}$, $f_2:=f_{s_2}$, $f_3:=f_{s_2s_1}$, and $f_4:=f_{s_2s_1s_2}$ form a basis of $V^I$ just as in the case of type IIb. The space $V_{\rho}^I$
is generated by $\phi_1=q^2f_1+q^2f_2-f_3-f_4$ and $\phi_2=(q^3+q^2)f_1+(q^2-q)f_2+(q^2-q)f_3-(q+1)f_4$. The space $V_{\rho}^{P_1}$ is
spanned by $\phi_2$, while the space $V_{\rho}^{P_2}$ is spanned $\phi_1$. There is no nontrivial $K$-invariant vector in $V_{\rho}$.

In terms of the chosen basis described above, the actions of $L_{I}$, $L_{P_1}$, $L_{P_2}$ are given by
$L_{I}(\phi_{1})=q\sigma(\pi)\phi_{1}-\sigma(\pi)\phi_{2}$,
$L_{I}(\phi_{2})=2q^{2}\sigma(\pi)\phi_{1}-q\sigma(\pi)\phi_{2}$;
$L_{P_{1}}(\phi_{2})=-q\sigma(\pi)\phi_{2}$, and
$L_{P_{2}}(\phi_{1})=-q^2\sigma^{2}(\pi)\phi_{1}$.
Hence the eigenvalues of $L_I$ on $V_{\rho}^I$ are $\pm\sqrt{-q}$. The eigenvalue of $L_{P_{1}}$ on $V_{\rho}^{P_1}$ is $-q\sigma(\pi)$, and
the eigenvalue of $L_{P_{2}}$ on $V_{\rho}^{P_2}$ is $-q^{2}$. The contribution to $R(u)$ from $\rho$ is $(1+q\sigma(\pi)u)^{-1}$.
As $L(\nu^{1/2}\xi\St_{\GL_2},\xi \nu^{-1/2}\sigma) \simeq \xi \rho$, its contribution to $R(u)$ is $(1-q\sigma(\pi)u)^{-1}$. Therefore the total contribution
to $R(u)$ from types Vb and Vc representations is $(1-q^2u^2)^{-\f{1}{2}m_{Vbc}}$.

\subsubsection{Type Vd}
Consider $\rho \simeq L(\nu\xi,\xi\rtimes\nu^{-1/2}\sigma)$ with $\sigma^2=\triv_{F^\times}$, then $\rho$ can be regarded as a subrepresentation in
$\nu^{-\f{1}{2}}\xi\triv_{GL_2} \rtimes \nu^{\f{1}{2}}\xi\sigma$. Similar to \S 5.6.1, we realize $\nu^{-\f{1}{2}}\xi\triv_{GL_2} \rtimes \nu^{\f{1}{2}}\xi\sigma$ in the standard induced model $V$ and identify $V_{\rho}$ as a subspace in $V$. Define a basis $f_1, \cdots, f_4$ of $V^I$ similarly. Then the image of $A(-\f{1}{2})$
restricted on $V^I$ characterizes $V_{\rho}^I$. It follows from \cite[Proposition 2.5 (ii)]{Sc3} that $V_{\rho}^I$ is
generated by $\phi_1=\f{1}{2}(q^2+1)f_1-\f{1}{2}(q-1)f_2-\f{1}{2}(q-1)f_3+f_4$ and $\phi_2=-\f{1}{2}(q^2-q)f_1+qf_2+qf_3+\f{1}{2}(q-1)f_4$.
Since $s_1 \in P_1$ and $s_2 \in P_2$, it is easy to observe that $V_{\rho}^{P_1}$ has as a basis  $\phi_1$ and $\phi_2$, while $V_{\rho}^{P_1}$ and $V_{\rho}^K$ are both 1-dimensional spaces generated by $\phi_1+\phi_2$.

The actions of $L_I$, $L_{P_1}$, $L_{P_2}$ on the chosen basis can be computed explicitly,
from which we find that the eigenvalues of $L_I$ on $V_{\rho}^I$ are $\pm\sqrt{-q^{3}}$; the eigenvalues of $L_{P_{1}}$ on $V_{\rho}^{P_1}$ are $\pm q^{2}\sigma(\pi)$; the eigenvalue of $L_{P_{2}}$ on $V_{\rho}^{P_2}$ is $-q^{3}$; and the zeros of $\det(I- A_{1}u+ qA_{2}u^{2}- q^{3}A_{1}u^{3}+ q^{6}Iu^{4})$ are $\pm q^{-2}\sigma^{-1}(\pi), \pm q^{-1}\sigma^{-1}(\pi)$. So the contribution from $\rho$ to $R(u)$ is $1-q^2u^2$, and the
total contribution to $R(u)$ from type Vd representations is $(1-q^2u^2)^{m_{Vd}}$.

\subsubsection{Type Va}
Suppose $\rho \simeq \delta([\xi,\nu\xi],\nu^{-1/2}\sigma)$ with $\sigma^2=\triv_{F^\times}$. As this is the only remaining case in Group V, the eigenvalues for the relevant adjacency operators on $V_{\rho}$ can be determined by applying results from sections 5.6.1, 5.6.2, and 5.1 with the inducing data $\chi_1$, $\chi_2$, $\sigma$ being $\nu \xi$, $\xi$, and $\nu^{-\f{1}{2}}\sigma$ respectively.
Then the eigenvalues of $L_I$ on $V_{\rho}^I$ are $\pm\sqrt{-q}$, and the eigenvalue of $L_{P_{2}}$ on $V_{\rho}^{P_2}$ is $-q$. Both of the spaces $V_{\rho}^{P_1}$ and $V_{\rho}^K$ are trivial. The contribution to $R(u)$ from type Va representations is trivial.

\subsection{Type VI representations}
Representations in type VI are constituents of $\nu \times \triv_{F^\times} \rtimes \nu^{-\f{1}{2}}\sigma$. When $\rho$ is of type VIb, VIc, or VId, a basis for
$V_{\rho}^I$ has been obtained in Schmidt \cite{Sc3} by studying appropriate intertwining operators. With the given basis of $V_{\rho}^I$ and the dimension count in Table 2, it becomes easy to find a basis for each of the spaces $V_{\rho}^{P_1}$, $V_{\rho}^{P_2}$, $V_{\rho}^K$. Then the action of each adjacency operator on $V_{\rho}$ can be described explicitly, so the corresponding eigenvalues can be computed directly. Once the eigenvalues from types VIb, VIc, VId are known, we then apply results in $\S5.1$ to determine the eigenvalues from the remaining type VIa case, just like the work done before. Since the argument is now routine,
we shall skip details, only present the final results.

\subsubsection{Type VIb} Consider $\rho \simeq \tau(T,\nu^{-1/2}\sigma)$ with $\sigma^2=\triv_{F^\times}$. Then $\rho$ can be regarded as a subrepresentation
in $\nu^{\f{1}{2}}1_{GL_{2}}\rtimes\nu^{-\frac{1}{2}}\sigma$. We realize $\nu^{\f{1}{2}}1_{GL_{2}}\rtimes\nu^{-\frac{1}{2}}\sigma$ in the standard induced
model $V$ and identify $V_{\rho}$ as a subspace in $V$. Let $f_1, \cdots, f_4$ be the functions in $V^I$ defined as in the case of type IIb. According to Schmidt \cite[Corollary 2.6]{Sc3}, the space $V_{\rho}^I$ is one-dimensional, generated by the function $q^2f_1-qf_2-qf_3+f_4$. This function is also a basis for $V^{P_1}$. Both of the spaces $V_{\rho}^{P_2}$ and $V_{\rho}^K$ are trivial. The eigenvalue of $L_I$ on $V_{\rho}^I$ is $-q\sigma(\pi)$ and the eigenvalue of $L_{P_1}$ on $V_{\rho}^{P_1}$ is $q\sigma(\pi)$. The total contribution to $R(u)$ from type VIb representations is trivial.

\subsubsection{Type VIc}
Suppose $\rho \simeq L(\nu^{\f{1}{2}}\St_{\GL_2},\nu^{-\f{1}{2}}\sigma)$ with $\sigma^2=\triv_{F^\times}$. Then $\rho$ can be regarded as a subrepresentation of
$1_{F^{\times}}\rtimes\sigma 1_{GSP_{2}}$. Realize $1_{F^{\times}}\rtimes\sigma1_{GSP_{2}}$ in the standard induced model $V$ and let $f_1, \cdots, f_4$ be
the functions in $V^I$ defined as in the case of type IIIb. By Schmidt \cite[Corollary 3.6]{Sc3}, $V_{\rho}^I$ is generated by the function $q^2f_1-qf_2-qf_3+f_4$, which also generates $V^{P_2}$. The spaces $V^{P_1}$ and $V^K$ are both trivial. The eigenvalue of $L_I$ on $V_{\rho}^I$ is $q\sigma(\pi)$ and the eigenvalue of $L_{P_2}$ on $V_{\rho}^{P_2}$ is $q^2$. The total contribution to $R(u)$ from type VIc representations is $(1-q^2u^2)^{-\f{1}{2}m_{VIc}}$.

\subsubsection{Type VId}
Suppose $\rho \simeq L(\nu,1_{F^\times}\rtimes\nu^{-\f{1}{2}}\sigma)$ with $\sigma^2=\triv_{F^\times}$. Then $\rho$ is a subrepresentation of $\nu^{-\frac{1}{2}}1_{GL_{2}}\rtimes\nu^{\frac{1}{2}}\sigma$, which again is realized in the standard induced model $V$. Let $f_1, \cdots, f_4$ be the functions
in $V_I$ defined as in the case of type IIb. Examining the kernel of the intertwining operator in \cite[Proposition 2.5]{Sc3} shows that the space $V_{\rho}^I$ has as a basis  $f_1+f_2$, $f_3+f_4$ and $qf_1+f_3$. The space $V_{\rho}^{P_1}$ is spanned by the functions $f_1+f_2+f_3+f_4$ and $q^2f_1+qf_2+qf_3+f_4$, while the space $V^{P_2}$ is spanned by $f_1+f_2$ and $f_3+f_4$. The 1-dimensional space $V^K$ is spanned by $f_1+f_2+f_3+f_4$.
The eigenvalues of $L_{I}$ on $V_{\rho}^I$ are $\pm\sqrt{q^{3}}$ and $q\sigma(\pi)$. The eigenvalues of $L_{P_1}$ on $V_{\rho}^I$ are
$q^{2}\sigma(\pi)$ and $q^{2}\sigma(\pi)$; and the eigenvalues of $L_{P_2}$ on $V_{\rho}^{P_2}$ are $q^{3}$ and $q^{2}$.
The zeroes of $\det(I- A_{1}u+ qA_{2}u^{2}- q^{3}A_{1}u^{3}+ q^{6}Iu^{4})$ arising from $\rho$ are $q^{-2}\sigma^{-1}(\pi), q^{-2}\sigma^{-1}(\pi), q^{-1}\sigma^{-1}(\pi),$ and $q^{-1}\sigma^{-1}(\pi)$. The total contribution to $R(u)$ from type VId representations is $(1-q^2u^2)^{\f{1}{2}m_{VId}}$.

\subsubsection{Type VIa} Suppose $\rho \simeq \tau(S,\nu^{-\f{1}{2}}\sigma)$ with $\sigma^2=\triv_{F^\times}$. The eigenvalues of $L_I$ on $V_{\rho}^I$ are $\pm\sqrt{q}$ and $ -q\sigma(\pi)$. The eigenvalue of $L_{P_1}$ on $V_{\rho}^{P_1}$ is $q\sigma(\pi)$; and the eigenvalue of $L_{P_2}$ on $V_{\rho}^{P_2}$ is $q$.
There is no nontrivial $K$-invariant vector. The total contribution to $R(u)$ from type VIa representations is trivial.

\section{Proof of the zeta identity}

Table 3 summarizes results from previous section. The column $L_I$ shows the eigenvalues of $L_I$ on each type of representation $(\rho, V_{\rho})$;
these eigenvalues are the reciprocals of the zeros of $\det(I-L_Iu)$ arising from $\rho$. Similarly, the column $L_{P_i}$ shows the reciprocals of the zeros
of $\det(I-L_{P_i}u)$ coming from $\rho$, for $i=1, 2$. The last column $A_1, A_2$ records the reciprocals of the zeros of $\det(I-A_{1}u+qA_{2}u^{2}-q^{3}A_{1}u^{3}+q^{6}Iu^{4})$ arising from $\rho$.
\begin{figure}
\begin{tabular}{|c|c|c|c|c|c|}
\hline
type& $L_I$& $L_{P_1}$&  $L_{P_2}$ & $A_1$, $A_2$ & condition\\
\hline

I & $\pm \sqrt{q^{2}\chi_{1}^{-1}(\pi)}$  & $q^{\frac{3}{2}}\chi_1\sigma(\pi)$ & $q^{2}\chi_1^{-1}(\pi)$ & $q^{\frac{3}{2}}\chi_1\sigma(\pi)$ & $\chi_1\chi_2\sigma^2 = \bf 1$\\
  \  & $\pm \sqrt{q^{2}\chi_{2}^{-1}(\pi)}$ & $q^{\frac{3}{2}}\chi_{2}\sigma(\pi)$ & $q^{2}\chi_{2}^{-1}(\pi)$ &  $q^{\frac{3}{2}}\chi_{2}\sigma(\pi)$& \\
  \  & $\pm \sqrt{q^{2}\chi_{1}(\pi)}$ & $q^{\frac{3}{2}}\chi_{1}\chi_{2}\sigma(\pi)$ & $q^{2}\chi_{1}(\pi)$ & $q^{\frac{3}{2}}\chi_{1}\chi_{2}\sigma(\pi)$ &\\
  \  & $\pm \sqrt{q^{2}\chi_{2}(\pi)}$ & $ q^{\frac{3}{2}}\sigma(\pi)$ & $q^{2}\chi_{2}(\pi)$ & $q^{\frac{3}{2}}\sigma(\pi)$& \\
 \hline

IIa & $\pm \sqrt{q^{\frac{3}{2}}\chi^{-1}(\pi)}$ & $q\chi\sigma(\pi)$ & $q^{\frac{3}{2}}\chi^{-1}(\pi)$ & none & $\chi^2\sigma^2 = \bf 1$\\
 \  & $\pm \sqrt{q^{\frac{3}{2}}\chi(\pi)}$ & \ & $ q^{\frac{3}{2}}\chi(\pi)$ & \ & \\
 \hline
IIb & $\pm \sqrt{q^{\frac{5}{2}}\chi^{-1}(\pi)}$ & $q^{\frac{3}{2}}\sigma^{-1}(\pi)$ & $q^{\frac{5}{2}}\chi^{-1}(\pi)$ &  $q^{\frac{3}{2}}\sigma^{-1}(\pi)$ & $\chi^2\sigma^2 = \bf 1$ \\
 \ & $\pm \sqrt{q^{\frac{5}{2}}\chi(\pi)}$ & $q^{\frac{3}{2}}\sigma(\pi)$ & $q^{\frac{5}{2}}\chi(\pi)$ &  $q^{\frac{3}{2}}\sigma(\pi)$  &\\
 \ & \  & $q^{2}\chi\sigma(\pi)$ & \  & $q^{2}\chi\sigma(\pi), q\chi\sigma(\pi)$ &\\
 \hline

IIIa & $\pm \sqrt{q}$ & $q\sigma^{-1}(\pi)$ & $q$ & none & $\chi \sigma^2 = \bf 1$\\
 \  & $-q\sigma^{-1}(\pi), -q\sigma(\pi)$ & $q\sigma(\pi)$ & \  & \ & \\
 \hline
IIIb & $\pm \sqrt{q^{3}}$ & $q^{2}\sigma^{-1}(\pi)$ & $q^{3}$ & $q^{2}\sigma^{-1}(\pi), q\sigma^{-1}(\pi) $ & $\chi \sigma^2 = \bf 1$ \\
 \  & $q\sigma^{-1}(\pi), q\sigma(\pi)$ & $q^{2}\sigma(\pi)$ & $q^{2}\sigma^{-2}(\pi), q^{2}\sigma^2(\pi)$  & $q^{2}\sigma(\pi), q\sigma(\pi)$ & \\
 \hline

IVa & $-\sigma(\pi)$  & none & none & none & $\sigma^2 = \bf 1$\\
\hline
IVd &  $q^{2}\sigma(\pi)$ & $q^{3}\sigma(\pi)$  & $q^{4}$ & $q^{3}\sigma(\pi), q^{2}\sigma(\pi)$& $\sigma^2 = \bf 1$\\
  \ & \  &\  &\  & $q\sigma(\pi), \sigma(\pi)$ &\\
\hline

Va & $\pm \sqrt{-q}$ & none & $-q$ & none & $\sigma^2 = \bf 1$\\
\hline
Vb & $\pm \sqrt{-q^{2}}$ & $-q\sigma(\pi)$ & $-q^{2}$ & none & $\sigma^2 = \bf 1$\\
\hline
Vc & $\pm \sqrt{-q^{2}}$ & $q\sigma(\pi)$ & $-q^{2}$ & none & $\sigma^2 = \bf 1$\\
\hline
Vd & $\pm \sqrt{-q^{3}}$ & $\pm q^{2}\sigma(\pi)$ & $-q^{3}$ & $\pm q^{2}\sigma(\pi), \pm q\sigma(\pi)$ & $\sigma^2 = \bf 1$\\
\hline

VIa & $\pm \sqrt{q}, -q\sigma(\pi)$ & $q\sigma(\pi)$ & $q$ & none & $\sigma^2 = \bf 1$\\
\hline
VIb & $-q\sigma(\pi)$ & $q\sigma(\pi)$ & none & none & $\sigma^2 = \bf 1$\\
\hline
VIc & $q\sigma(\pi)$ & none & $q^{2}$ & none & $\sigma^2 = \bf 1$\\
\hline
VId & $\pm \sqrt{q^{3}}, q\sigma(\pi)$ & $q^{2}\sigma(\pi), q^{2}\sigma(\pi)$& $q^{3}, q^{2}$ & $q^{2}\sigma(\pi), q^{2}\sigma(\pi)$ & $\sigma^2 = \bf 1$\\
\ & \ & \ & \ & $q\sigma(\pi), q\sigma(\pi)$ & \\
\hline

\end{tabular}

\begin{center}
Table 3
\end{center}
\end{figure}
\smallskip

We finish the proofs of Corollary 4.3 and Theorem 4.2.
By the computations in $\S 5$, the net contribution to $R(u)$ from each type of representations in Table 1 can be listed as follows.
\begin{itemize}
\item Type I: \quad 1;

\item Type II: \quad $(1-q^2u^2)^{-\f{1}{2}m_{IIa}+\f{1}{2}m_{IIb}}$;

\item Type III: \quad 1;

\item Type IV: \quad $(1-u^2)^{\chi(X_{\Gamma})}(1-q^2u^2)$;

\item Type V: \quad $(1-q^2u^2)^{-\f{1}{2}m_{Vbc}+m_{Vd}}$;

\item Type VI: \quad $(1-q^2u^2)^{-\f{1}{2}m_{VIc}+\f{1}{2}m_{VId}}$.

\end{itemize}

\noindent Taking product of the contributions to $R(u)$ over all six types of representations gives
$$R(u)= (1-u^2)^{\chi(X_{\Gamma})}(1-q^2u^2)^{\f{1}{2}m},$$ where $m=-m_{IIa}+m_{IIb}-m_{Vbc}+2m_{Vd}-m_{VIc}+m_{VId}+2.$
By Proposition \ref{multi} (2), this gives $$R(u)= (1-u^2)^{\chi(X_{\Gamma})}(1-q^2u^2)^{-(q^2-1)N_p}.$$
In other words, $$\frac{(1-u^2)^{\chi(X_{\Gamma})}(1-q^2 u^2)^{-(q^2-1)N_p}}{\det(I-A_1 u+qA_2 u^2-q^3 A_1 u^3+q^6I u^4)}
= \frac{\det(I-L_I u)}{\det(I-L_{P_1} u)\det(I-L_{P_2}u^2)},$$
which completes the proof of Corollary 4.3 and hence Theorem 4.2 follows.

\bigskip

\section{Ramanujan complexes}
The notion of Ramanujan complexes was introduced in Li \cite{Li1} and Lubotzky-Samuels-Vishne \cite{LSV2} for finite quotient complexes arising from the Bruhat-Tits building $\B_n$ associated to $\SL_n(F)$. It generalizes the concept of Ramanujan graphs, which corresponds to the case $n=2$. Originally the definition was given in terms of eigenvalues of vertex adjacency operators. From the representation-theoretic point of view, it can be rephrased as follows (cf. \cite[Proposition 1.5]{LSV2}). For a discrete torsion-free co-compact subgroup $\Gamma$ of $\PGL_n(F)$, the finite complex $\Gamma \bb \B_n$ is Ramanujan if and only
if all the irreducible unramified infinite-dimensional representations that occur in $L^2(\Gamma \bb \PGL_n(F))$ are tempered. In \cite{KLW}, equivalent conditions for $\Gamma \bb \B_3$ being Ramanujan are given in terms of eigenvalues of vertex, edge, and chamber adjacency operators respectively.

We extend the notion of Ramanujan complexes to the symplectic case. Given a discrete torsion-free cocompact-mod-center subgroup $\Gamma$ of $G$ as before, define $X_{\Gamma}$ to be Ramanujan if all the irreducible unramified infinite-dimensional representations that occur in $L^2(\Gamma \bb G/Z)$ are tempered. Since only infinite-dimensional representations are concerned in the definition, the zeros of each determinant listed in Table 3 arising from 1-dimensional representations (type IVd) are called trivial zeros, whereas those arising from infinite-dimensional representations are nontrivial.

The unramified representations listed in Table 1 are those with $\dim V^K=1$; they are of types I, IIb, IIIb, Vd, or VId after excluding the finite-dimensional case. These representations are tempered only when they are of type I with $\chi_1$, $\chi_2$ and $\sigma$ being unitary.
Using the data from Table 1 and Table 3, one easily checks that if $X_{\Gamma}$ is Ramanujan, then all the nontrivial zeros of $\det(I-A_{1}u+qA_{2}u^{2}-q^{3}A_{1}u^{3}+q^{6}Iu^{4})$ have absolute value $q^{-\frac{3}{2}}$. On the other hand, if there is a nontempered unramified
infinite-dimensional representation occurring in $L^2(\Gamma \bb G/Z)$ so that $X_{\Gamma}$ is not Ramanujan, the determinant
$\det(I-A_{1}u+qA_{2}u^{2}-q^{3}A_{1}u^{3}+q^{6}Iu^{4})$ will have a nontrivial zero with absolute value other than $q^{-\frac{3}{2}}$.
Similar observations applied to the operators $L_{P_1}$, $L_{P_2}$, and $L_{I}$ give the following analog of Theorem 2 in \cite{KLW}.

\begin{theorem}\label{ram}

The following statements are equivalent:

$(1)$ $X_{\Gamma}$ is Ramanujan.

$(2)$ All the nontrivial zeros of $\det(I-A_{1}u+qA_{2}u^{2}-q^{3}A_{1}u^{3}+q^{6}Iu^{4})$ have absolute value $q^{-\frac{3}{2}}$.

$(3)$ All the nontrivial zeros of $\det(I - L_{P_{1}} u)$ have absolute values $q^{-\frac{3}{2}}$ or $q^{-1}$.

$(4)$ All the nontrivial zeros of $\det(I - L_{P_{2}} u)$ have absolute values $q^{\alpha}$, where $-2 \le \alpha \le -1$.

$(5)$ All the nontrivial zeros of $\det(I - L_{I} u)$ have absolute values $1$ or $q^{\beta}$, where $-1 \le \beta \le -\frac{1}{2}$.

\end{theorem}

\end{document}